\documentclass[10pt]{amsart}
\usepackage{amssymb,MnSymbol}
\usepackage{amsthm,amsmath}
\usepackage{epic,cite}
\usepackage[justification=centering]{caption}
\usepackage{amssymb,MnSymbol}%,times}
\usepackage{amsthm,amsmath,enumerate,color}
\usepackage{cite}
\usepackage[utf8]{inputenc}
\usepackage{graphicx, subfig}
\usepackage{tikz}
\usepackage{caption}
\usepackage{cite}

\title[Visualization of associative quasitrivial nondecreasing operations]{Visual characterization of associative quasitrivial nondecreasing operations on finite chains}

\author{Gergely Kiss}
\address{Mathematics Research Unit, University of Luxembourg, Maison du Nombre, 6, avenue de la Fonte, L-4364 Esch-sur-Alzette, Luxembourg}
\email{gergely.kiss[at]uni.lu}


\usepackage{amssymb,MnSymbol}%,times}
\usepackage{amsthm,amsmath,enumerate,color}
\usepackage{cite}
\usepackage[utf8]{inputenc}
\usepackage{graphicx, subfig}
\usepackage{tikz}
\usepackage{caption}
\usepackage{cite}

\theoremstyle{plain}
\newtheorem{theorem}{Theorem}[section]% Supprimer [section] pour une num?rotation lin?aire
\newtheorem{lemma}[theorem]{Lemma}
\newtheorem{proposition}[theorem]{Proposition}
\newtheorem{corollary}[theorem]{Corollary}

\newtheorem{obs}[theorem]{Observation}
\theoremstyle{definition}
\newtheorem{definition}[theorem]{Definition}
\newtheorem{example}[theorem]{Example}

\theoremstyle{remark}
\newtheorem{remark}{Remark}

\begin{document}
%\title{Visual characterization of associative quasitrivial monotone functions defined on a finite chain}
%\author{Jimmy Devillet and Gergely Kiss}

%\author[G. Kiss]{Gergely Kiss}
%\address{Mathematics Research Unit, FSTC, University of Luxembourg \\
%6, rue Coudenhove-Kalergi, L-1359 Luxembourg, Luxembourg}
%\email{gergely.kiss[at]uni.lu }

%\author[G. Somlai]{G\'abor Somlai}
%\address{Mathematics Research Unit, FSTC, University of Luxembourg \\a
%6, rue Coudenhove-Kalergi, L-1359 Luxembourg, Luxembourg}
%\email{jean-luc.marichal[at]uni.lu }

%

\date{\today}

%\date{\today,\currenttime}

\maketitle

\begin{abstract}
In this paper we provide visual characterization of associative
quasitrivial nondecreasing operations on finite chains. We also
provide a characterization of bisymmetric quasitrivial nondecreasing
binary operations on finite chains. Finally, we estimate the number
of operations belonging to the previous classes.
\end{abstract}

%%%%%%%%%%%%%%%%%%%%%%%%%%%%%%%%%%%%%%%%%%%%%%%%%%%%%%%%%%%%%%%%%%%%%%%%%%%%%%%%
\section{Introduction}\label{s1}
%In this paper we present a characterization of associative quasitrivial monotone binary operations on finite chains. The analogue result for $[0,1]$ was made by \cite{Martin2003}.
%In the symmetric case the following result was proved. ..

%In \cite[Theorem 3.3 and Corollary 3.4]{DKM}  the following results were proved.
%\begin{theorem}\label{propQS}
%Let $X$ be a totally ordered set. Let $F\colon X^k\to X$ be a
%quasitrivial, symmetric, nondecreasing, associative operation. Then
%$F$ is reducible. More precisely, $F$ is derived from $G\colon
%X^2\to X$ defined by \begin{equation}\label{eqxy}
%G(x,y)=F((k-1)\cdot x,y)=F(x,(k-1)\cdot y).
%\end{equation}
%In this case $F$ can be calculated by the following %formula.
%\begin{equation}\label{eqext1}
%F(x_1, \dots, x_k)=G(\wedge_{i=1}^k x_i, \vee_{i=1}^k %x_i).
%\end{equation}
%\end{theorem}
%It is easy to see that operation $G$ defined by %\eqref{eqxy} is also
%quasitrivial, symmetric and nondecreasing.
%e note that as in \cite[Theorem 3.3]{DKM} it can be %shown  using symmetry
%that in this case
%\begin{equation}
%F(x_1, \dots, x_k)=G(\wedge_{i=1}^k x_i, \vee_{i=1}^k %x_i).
%\end{equation}
%We note that equation \eqref{eqext1} means that
The study of aggregation operations defined on finite ordinal scales
(i.e, finite chains) have been in the center of interest in the last
decades, e.g., \cite{C,B,Fo,Ma1, Ma2, Ma3, Ma4, May2, May, Li, RA,
Su1, Su2}. Among these operations, discrete uninorms has an
important role in fuzzy logic and decision making \cite{Be, B1,  B2,
F}.

In this paper we investigate associative quasitrivial nondecreasing
operations on finite chains.
In \cite{Beats2009, Jimmy2017, DKM} idempotent discrete uninorms (i.e. idempotent symmetric nondecreasing associative operations with neutral elements defined on finite chains) have been characterized. Since every idempotent uninorm is quasitrivial (see e.g. \cite{Czogala1984}), in some sense this paper is a continuation of these works %\cite{Beats2009, Jimmy2017, DKM}
where we eliminate the assumption of symmetry of the operations.

 %that is motivated by \cite{Beats2009, Jimmy2017, DKM}.
%Since every idempotent uninorms are quasitrivial, in some sense this paper is a continuation of the works \cite{Beats2009, Jimmy2017, DKM} for finite chains where we eliminate the assumption of symmetry of the operations.

Now we recall the analogue results for the unit interval $[0,1]$ as
follows. Czoga\l a-Drewniak proved in \cite{Czogala1984} that an
associative monotonic idempotent operation with neutral element is a
combination of minimum and maximum, and thus these are quasitrivial.
Martin, Mayor and Torrens in \cite{Martin2003} gave a complete
characterization of associative quasitrivial nondecreasing
operations on $[0,1]$. A refinement of their argument can be found
in \cite{Beats2010}. (For the multivariable generalization of these
results see \cite{GG2016}.) We note that in \cite{Beats2009} the
analogue of the result of Czoga\l a-Drewniak for finite chains has
been provided assuming of symmetry of such operations.

The study of $n$-ary operations $F\colon X^n\to X$ satisfying the associativity property (see Definition \ref{dbase}) stemmed from the work of D\"ornte \cite{Dor28} and Post \cite{Pos40}. In \cite{DudMuk96, DM} the reducibility (see Definition \ref{defred}) of associative $n$-ary operations have been studied by adjoining neutral elements. In \cite{A} a complete characterization of quasitrivial associative $n$-ary operations have been presented. In \cite{DKM} the quasitrivial symmetric nondecreasing associative $n$-ary operations defined on chains have been characterized. Recently, in \cite{KS} it was proved that associative idempotent nondecreasing $n$-ary operations defined on any chain are reducible. Using reducibility (see Theorem \ref{t2}) a characterization of associative quasitrivial nondecreasing $n$-ary operations for any $2\le n\in \mathbb{N}$ can be obtained automatically by %  Since every quasitrivial operation is idempotent,
a characterization of associative quasitrivial nondecreasing binary
operations.

The paper is organized as follows. In Section \ref{s2} we present the most important definitions. %At the beginning of Section \ref{s3} by recalling \cite[Theorem 4.8]{KS}, which states that every associative idempotent nondecreasing $n$-ary operation is derived from a binary one. Thus,
In Section \ref{s3}, we recall (\cite[Theorem 4.8]{KS}) the
reducibility of associative idempotent nondecreasing $n$-ary
operations and, hence, in the sequel we mainly focus on the binary
case. We introduce the basic concept of visualization for
quasitrivial monotone binary operations and present some preliminary
results due to this concept. Here we discuss an important visual
test of non-associativity (Lemma \ref{lpic}). Section \ref{s4} is
devoted to the visual characterization of associative quasitrivial
nondecreasing operations with so-called 'downward-right paths'
(Theorems \ref{tfo1} and \ref{thmchar}). We also present an
Algorithm which provides the contour plot of any associative
quasitrivial nondecreasing operation.
In Section \ref{s5} we characterize the bisymmetric quasitrivial nondecreasing binary operations (Theorem \ref{tbfo1}). %In Subsection \ref{s5.1} we show that the analogue of \cite[Theorem 4.8]{KS} does not hold for bisymmetric quasitrivial nondecreasing $n$-ary operations (Example \ref{exbi}). This answers the question whether a quasitrivial, bisymmetric operation is associative (see \cite[Remark 10.(b)]{DKM}). In Subsection \ref{s5.2} we provide some special classes of $n$-ary bisymmetric operations where a reduction to the binary case can be guaranteed (Corollary \ref{cbired}).
In Section \ref{s6} we calculate the number of associative
quasitrivial nondecreasing operations defined on a finite chain of
given size with and also without the assumption of the existence of
neutral elements (Theorem \ref{tnumb}).
We get similar estimations for the number of bisymmetric  quasitrivial nondecreasing binary operations defined on a finite chain of given size. %with and also without the assumption of neutral elements
(Proposition \ref{binumb}). In Section \ref{s7} we present some
problems for further investigation. Finally, using a slight
modification of the proof of \cite[Theorem 3.2]{KS}, in the Appendix
we show that every associative quasitrivial monotonic $n$-ary
operations are nondecreasing.

\section{Definition}\label{s2}
%\section{Definition of the binary case}\label{s2}
Here we present the basic definitions and some preliminary results.
First we introduce the following simplification. For any integer
$l\geq 0$ and any $x\in X$, we set $l\cdot x= x,\ldots,x$ ($l$
times). For instance, we have $F(3\cdot x_1,2\cdot
x_2)=F(x_1,x_1,x_1,x_2,x_2)$.
\begin{definition}\label{dbase}
Let $X$ be an arbitrary nonempty set. A operation  $F:X^n\to X$ is
called
\begin{itemize}
\item \emph{idempotent} if $F(n\cdot x)=x$ for all $x\in X$;
\item \emph{quasitrivial} (or \emph{conservative}) %\emph{selective})
if $$F(x_1,\ldots,x_n)\in\{x_1,\ldots,x_n\}$$ for all
$x_1,\ldots,x_n\in X$;
\item \emph{($n$-ary) associative} if
\begin{eqnarray*}\label{assoc}
\lefteqn{F(x_1,\ldots,x_{i-1},F(x_i,\ldots,x_{i+n-1}),x_{i+n},\ldots,x_{2n-1})}\\
&=&
F(x_1,\ldots,x_i,F(x_{i+1},\ldots,x_{i+n}),x_{i+n+1},\ldots,x_{2n-1})
\end{eqnarray*}
for all $x_1,\ldots,x_{2n-1}\in X$ and all $i\in \{1, \dots, n-1\}$;
\item\emph{($n$-ary) bisymmetric} if
$$
F(F(\mathbf{r}_1),\ldots,F(\mathbf{r}_n)) ~=~
F(F(\mathbf{c}_1),\ldots,F(\mathbf{c}_n))
$$
for all $n\times n$ matrices $[\mathbf{r}_1 ~\cdots
~\mathbf{r}_n]=[\mathbf{c}_1 ~\cdots ~\mathbf{c}_n]^T\in X^{n\times
n}$.
%\end{equation}
\end{itemize}

We say that $F:X^n\to X$ has a {\it neutral element} $e\in X$ if for
all $x\in X$ and all $i\in \{1,\dots, n\}$
$$
F((i-1)\cdot e,x,(n-i)\cdot e) ~=~ x.
$$

Hereinafter we simply write that an $n$-ary operation is associative
or bisymmetric if the context clarifies the number of its variables.
We also note that if $n=2$ we get the binary definition of
associativity, quasitriviality, idempotency, and neutral element
property.

%Let $(X, \ge)$ be a chain. An  operation  $F:X^2\to X$ is called

%\end{definition}

Let $(X,\leq)$ be a nonempty chain (i.e, a totally ordered set). An
operation $F\colon X^n\to X$ is said to be
\begin{itemize}
\item {\it nondecreasing} (resp. {\it nonincreasing}) %(w.r.t.\ $\leq$)
if $$F(x_1,\ldots,x_n)\leq F(x'_1,\ldots,x'_n) \ \ \ (\textrm{resp.
} F(x_1,\ldots,x_n)\geq F(x'_1,\ldots,x'_n))$$ whenever $x_i\leq
x'_i$ for all $i\in \{1,\dots, n\}$,
\item {\it monotone in the $i$-th variable} if
for all fixed elements $a_1,\dots a_{i-1}, a_{i+1}, \dots, a_n$ of
$X$, the $1$-ary function defined as $$f_i(x):=F(a_1,\dots,
a_{i-1},x,a_{i+1},\dots, a_n)$$ is nondecreasing or nonincreasing.
\item {\it monotone} if it is
monotone in each of its variables.

% \item {\it nonincreasing} (or {\it order-reversing}) %(w.r.t.\ $\leq$)
%if $F(x_1,\ldots,x_n)\leq F(x'_1,\ldots,x'_n)$ whenever $x_i\leq x'_i$ for all $i\in \{1,\dots, n\}$.
%\end{itemize}
\end{itemize}
\end{definition}

\begin{definition}\label{defred}

%\begin{enumerate}
%\item
We say that $F:X^n\to X$ {\it is derived from} a binary operation
$G:X^2\to X$ if $F$ can be written of the form
\begin{equation}\label{eqjdef}
    F(x_1, \dots, x_n)=x_1\circ \dots
\circ x_n,
\end{equation} where $x\circ y=G(x,y)$. It is easy to see that $G$ is associative (and $F$ is $n$-ary associative) if and only if \eqref{eqjdef}
is well-defined. If such a $G$ exists, then we say that $F$ is {\it
reducible}.
%We also say that $F:$
%\item
%\end{enumerate}
\end{definition}

We denote the {\it diagonal} of $X^2$ by $\Delta_X=\{(x,x):x\in
X\}$.

\begin{definition}
Let $L_k$ denote $\{1,\dots, k\}$ endowed with the natural ordering
$(\le)$.
\end{definition}

Then $L_k$ is a finite chain.  Moreover, every finite chain with $k$
element can be identified with $L_k$ and the domain of an
$n$-variable operation defined on a finite chain can be identified
with $\underbrace{L_k\times\cdots\times L_k}_n=(L_k)^n$ for some
$k\in \mathbb{N}$.

For an arbitrary poset $(X, \le)$ and $a\le b\in X$ we denote the elements between $a$ and $b$ by $[a,b]\subseteq X$. In particular, for $L_k$ % this is the set of natural numbers between $a$ and $b$.
%Formally,
$$[a,b]=\{m\in L_k: a \le m \le b\}.$$ %\textrm{ or } a \le s \le b\}.$

We also introduce the lattice notion of the minimum ($\wedge$) and
the maximum ($\vee$) as follows
$$x_1\wedge\dots\wedge x_n=\wedge_{i=1}^n x_i=\min\{x_1, \dots, x_n\},$$
$$x_1\vee\dots\vee x_n=\vee_{i=1}^n x_i=\max\{x_1, \dots, x_n\}.$$

The binary operations Proj$_x$ and Proj$_y$ denote the projection to
first and the second coordinate, respectively. Namely, Proj$_x(x, y)
= x$ and Proj$_y(x, y) = y$ for all $x, y \in X.$

\section{Basic concept and preliminary results}\label{s3}
The following general result was published as \cite[Theorem 4.8]{KS}
recently.

\begin{theorem}\label{t2}
Let $X$ be a nonempty chain and $F:X^n\to X$ $(n\ge 2)$ be an
associative idempotent nondecreasing operation. Then there exists
uniquely an associative idempotent nondecreasing binary operation
$G:X^2\to X$ such that $F$ is derived from $G$. Moreover, $G$ can be
defined by
\begin{equation}\label{eqgf}
G(a,b)=F(a, (n-1)\cdot b)=F((n-1)\cdot a, b)~~~~(a,b\in
X).\end{equation}
\end{theorem}

\begin{remark}\label{rem2}
By the definition \eqref{eqgf} of $G$, it is clear that if $F$ is
quasitrivial, then $G$ is also.
\end{remark}

According to Theorem \ref{t2} and Remark \ref{rem2},  a characterization of associative quasitrivial nondecreasing binary operations automatically implies a characterization for the $n$-ary case. Therefore, from now on we deal with the binary case ($n=2$). %That is what we do in Section \ref{s4} for binary operations that are defined on finite chains.

\subsection{Visualization of binary operations}\label{s3.1}
In this section we prove and reprove basic properties of
quasitrivial associative nondecreasing binary operations in the
spirit of visualization.

\begin{obs}\label{l1}
Let $X$ be a nonempty chain and let $F:X^2\to X$ be a quasitrivial
monotone operation. If $F(x,t)=x$, then $F(x,s)=x$ for every $s\in
[x\wedge t, x\vee t]$ . Similarly, if $F(x,t)=t$, then $F(s,t)=t$
for every $s\in [x\wedge t, x\vee t]$.
\end{obs}

%\begin{proof} Obvious.
%The operation $F$ is quasitrivial, thus idempotent. Monotonicity implies the statement.% as follows.
%
%Without loss of generality we can assume that $x\le t$. %\in X$.
%Using monotonicity for every $s\in [x,t]$ one of the following holds  \begin{align*}
%x=F(x,t)= &F(x,s)= F(x,x),  \\  &\textrm{or}\\
%t=F(x,t)= &F(s,t)= F(t,t).
%\end{align*}
%\end{proof}
A {\it level-set} of $F$ is a set of vertices of $L_k^2$ where $F$
has the same value. The {\it contour plot} of $F$ can be visualized
by connecting the closest elements of the level-sets of $F$ by line
segments.
According to Observation \ref{l1}, this contour plot can be drawn using only horizontal and vertical line segments starting from the diagonal (as in Figure 1.). %Thus these segments represent the level sets of $F$. This picture that we call the $\it{contour plot}$ of $F$.
%These lines (restricted to $L_k$) represent the different values of $F$.
It is clear that these lines do not cross each other by the
monotonicity of $F$.
\begin{figure}[ht]
\centering
\begin{tikzpicture}
 %C/.style = {circle,thick,draw, inner sep=2pt},
\draw (0,0)-- (0,4)--(4,4)--(4,0) -- cycle; \draw (0,0)--(4,4);
\draw[very thick] (1,1)--(3,1);
%\draw[very thick] (2,2)--(1,2);
\draw[very thick] (3,3)--(3,2);
%\node[] at (1,-0.3) {x};
\node[] at (-0.3,1) {y};
%\node[] at (2,-0.3) {y};
\node[] at (-0.3,2) {z}; \node[] at (3,-0.3) {x};
%\node[] at (-0.3,3) {z};
\draw[step=1cm, gray, very thin](0,0) grid (4,4);
%\node[] at (2,-1) {Case 1/b: $y<z<x$ };
%\filldraw (1,1) circle[radius=1.5pt];
\end{tikzpicture}
\caption{$F(x,y)=y$ and $F(x,z)=x$ }
  \label{figbase}
  \end{figure}
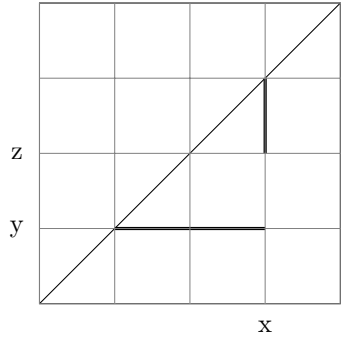

%Here we note that the previous observation can be deduced easily from \cite[Proposition 6]{Jimmy2017}.
As a consequence we get the following.
\begin{corollary}\label{cnem}
Let $X$ be a nonempty chain and $F:X^2 \to X$ be a quasitrivial
operation. $$F \textrm{ is monotone } \Longleftrightarrow F \textrm{
is nondecreasing}.$$
\end{corollary}

\begin{proof}
We only need to prove that every monotone quasitrivial operation is
nondecreasing.

As an easy consequence of Observation \ref{l1} and the
quasitriviality  of $F$, we have $F(s,x)\le F(t,x)$ and $F(x,t)\le
F(s,t)$ for any $x,s,t\in X$ that satisfies $s\in[x,t]$.
%If $x\ge t$ then $F(x)$
This implies that $F$ is nondecreasing in the first variable.
Similar argument shows the statement for the second variable.
%Let $F$ be decreasing operation fixing the first variable in $x\in L_k$. Thus, there exist $y,z\in L_k$ $(y<z)$ and $F(x,y)>F(x,z)$.
%By the previous observation it is clear that each element  $(x,y)\in L_k\times L_k$ is connected to the diagonal with a horizontal or a vertical line that does not cross each other. %This means that $F$ is constant in $(x,y)$ in one of its variables. Without loss of generality we assume that $x<y$. Let $$s=\max{F()}$$Thus it cannot be nondecreasing.
%Thus for any $s\in [x,t]$ $F(x,s)=x$ or ($F(x,s)=s$ and $F(x,t)=t$). This simple gives that $F$ cannot be decreasing. Indeed,
%Now we show that every coordinate operation if $F$ is nondecreasing that
%This also finishes the proof of the statement.

%If $F(x,t)=x$, then  $F(x,s)=x$ by Observation \ref{l1}.  In the case $F(x,t)=t$, we know that $F(x,s)=s$ or $x$ thus $F(x,s)\le F(x,t)$. The case $x>s>t$ can be handled similarly.
%The same argument works for the second variable.
\end{proof}
%We note that the analogue statement holds for $n\ge 3$ (see ...).

\begin{remark}\label{rnem}
%Proposition
The analogue of Corollary \ref{cnem} holds whenever $n> 2$. The
proof is essentially the same as the proof of \cite[Theorem
3.10]{KS}. Thus we present it in Appendix A.
\end{remark}
In the sequel we are dealing with associative, quasitrivial and
nondecreasing operations.
%%%%%% Ide kell meg a kesobbi hivatkozas

There are several know forms of the following proposition. This type
of results was first proved in  \cite{Martin2003}. The form as
stated here is \cite[Proposition 18]{Jimmy2017}.

 \begin{proposition}\label{prop:eqv}
 Let $X$ be an arbitrary nonempty set and %let $\Delta_X=\{(x,x):x\in X\}$.
 let $F:X^2 \to X$ be a quasitrivial operation. Then the following assertions are equivalent.
 \begin{enumerate}
     \item[(i)] $F$ is {\bf not} associative.
     \item[(ii)] There exist pairwise distinct $x,y,z\in X$ such that $F(x,y), F(x,z), F(y,z)$ are pairwise distinct.
     \item[(iii)] There exists a rectangle in $X^2$ such that one of the vertices is on $\Delta_X$ and the three remaining vertices are in $X^2\setminus \Delta_X$ and pairwise disconnected.
 \end{enumerate}
 \end{proposition}

%\begin{remark}\label{remiv}
%\end{remark}
%\end{remark}
%For the sake of completeness we present also a short proof.
%\begin{lemma}\label{l2}
%Let $F:X^2\to X$ be a
% nondecreasing quasitrivial operation , then $F$ is not associative %, i.e:
%\begin{equation}\label{cdt:06}
%   F(F(x,y),z)=F(x,F(y,z))
%\end{equation}
%if and only if there are some $x,y,z\in X$ which satisfies
%    \begin{equation}\label{eq1}
%    F(x,y)=x, F(x,z)=z, F(y,z)=y,
%    \end{equation}
%    or
%    \begin{equation}\label{eq2}
%    F(x,y)=y, F(y,z)=z, F(x,z)=x.
%    \end{equation}
%\end{lemma}
%\begin{proof}
%($\Longrightarrow$): We prove it by contraposition. Direct calculation shows that if there exist such $x,y,z\in X$ satisfying \eqref{eq1} and \eqref{eq2}, then $F$ is not associative. \\
%($\Longleftarrow$): If there are no such $x,y,z\in X$ satisfying \eqref{eq1} and \eqref{eq2} then
%one of the following cases holds \begin{align*}
%F(x,y)=x, F(x,z)=x;\\
%F(x,y)=y, F(y,z)=y;\\
%F(x,z)=z, F(y,z)=z.
%\end{align*}
%In the first case we get
%F(F(x,y),z)=F(x,z)=x \text{ and }  F(x,F(y,z))=F(x, y\text{ or } z)=x.\]
% In the other cases similar argument shows that $F$ is associative.
%\end{proof}

Now %for nondecreasing, quasitrivial operations
we present a visual version of the % which is also an equivalent form of the
previous statement if $F$ is nondecreasing.
\begin{lemma}\label{lpic}
Let $X$ be chain and $F:X^2\to X$ a quasitrivial, nondecreasing
operation. Then $F$ is {\bf not} associative if and only if there
are pairwise distinct elements $x,y,z\in X$  that give one of the
following pictures.
\begin{figure}[ht]
\centering
\begin{tikzpicture}
 %C/.style = {circle,thick,draw, inner sep=2pt},
\draw (0,0)-- (0,2)--(2,2)--(2,0) -- cycle; \draw (0,0)--(2,2);
\draw[very thick] (0.4,0.4)--(1,0.4); \draw[very thick]
(1,1)--(1,1.7); \draw[very thick] (1.7,1.7)--(1.7,0.4); \node[] at
(0.4,-0.3) {z}; \node[] at (-0.3,0.4) {z}; \node[] at (1,-0.3) {x};
\node[] at (-0.3,1) {x}; \node[] at (-0.3,1.7) {y}; \node[] at
(1.7,-0.3) {y}; \node[] at (1,-1) {(a)};

\draw (3,0)-- (3,2)--(5,2)--(5,0) -- cycle; \draw (3,0)--(5,2);
\draw[very thick] (3.4,0.4)--(3.4,1.7); \draw[very thick]
(4,1)--(4,0.4); \draw[very thick] (4.7,1.7)--(4,1.7); \node[] at
(3.4,-0.3) {y}; \node[] at (2.7,0.4) {y}; \node[] at (4,-0.3) {x};
\node[] at (2.7,1) {x}; \node[] at (2.7,1.7) {z}; \node[] at
(4.7,-0.3) {z}; \node[] at (4,-1) {(b)};

\draw (6,0)-- (6,2)--(8,2)--(8,0) -- cycle; \draw (6,0)--(8,2);
\draw[very thick] (6.4,0.4)--(6.4,1); \draw[very thick]
(7,1)--(7.7,1); \draw[very thick] (7.7,1.7)--(6.4,1.7); \node[] at
(6.4,-0.3) {x}; \node[] at (5.7,0.4) {x}; \node[] at (7,-0.3) {z};
\node[] at (5.7,1) {z}; \node[] at (5.7,1.7) {y}; \node[] at
(7.7,-0.3) {y}; \node[] at (7,-1) {(c)};

\draw (9,0)-- (9,2)--(11,2)--(11,0) -- cycle; \draw (9,0)--(11,2);
\draw[very thick] (9.4,0.4)--(10.7,0.4); \draw[very thick]
(10,1)--(9.4,1); \draw[very thick] (10.7,1.7)--(10.7,1); \node[] at
(9.4,-0.3) {y}; \node[] at (8.7,0.4) {y}; \node[] at (10,-0.3) {z};
\node[] at (8.7,1) {z}; \node[] at (8.7,1.7) {x}; \node[] at
(10.7,-0.3) {x}; \node[] at (10,-1) {(d)};
%\draw[step=1cm, gray, very thin](0,0) grid (4,4);
\end{tikzpicture}
\caption{Four pictures that guarantee the non-associativity of $F$}
\label{fig111}
\end{figure}
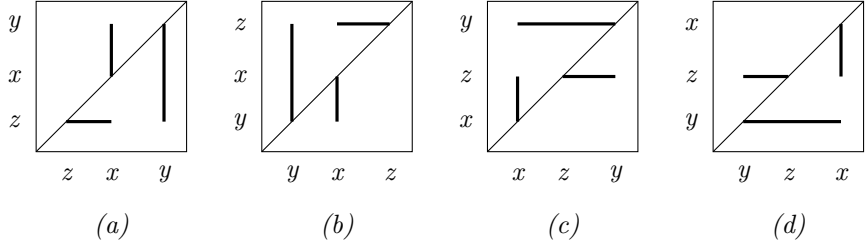
\end{lemma}

\begin{proof}
By Proposition \ref{prop:eqv}, $F$ is not associative if and only if
there exists distinct $x,y,z\in X$ satisfying one of the following
cases:
    \begin{equation}\label{eq1}
    F(x,y)=x, F(x,z)=z, F(y,z)=y~~~~ \textrm{(Case 1),}
    \end{equation}
    {\centering or}
    \begin{equation}\label{eq2}
    F(x,y)=y, F(y,z)=z, F(x,z)=x~~~~\textrm{(Case 2)} .
    \end{equation}

%\eqref{cdt:06}.
Since $x,y,z\in X$ pairwise distinct elements, they can be ordered
in 6 possible configuration of type $x<y<z$.
For each case either \eqref{eq1} or \eqref{eq2} holds. Therefore we have 12 configurations as possible realizations of Case 1 or Case 2. % of \eqref{eq1} or \eqref{eq2}. %Hereinafter we call Case 1 or Case 2 if  or equation \eqref{eq2} is satisfied, respectively.

% Using our notation and Observation \ref{l1} we can draw a picture from this situation.
Let us consider Case 1 (when equation \eqref{eq1} holds) and assume
$x<y<z$. This implies the situation of Figure \ref{f1}.
\begin{figure}[ht]
\centering
\begin{tikzpicture}
 %C/.style = {circle,thick,draw, inner sep=2pt},
\draw (0,0)-- (0,3)--(3,3)--(3,0) -- cycle; \draw (0,0)--(3,3);
\draw[very thick] (0.75,0.75)--(0.75,1.5); \draw[very thick]
(1.5,1.5)--(1.5,2.25); \draw[very thick] (2.25,2.25)--(0.75,2.25);
\filldraw[red] (1.5,2.25) circle (0.1cm); \node[] at (0.75,-0.3)
{x}; \node[] at (-0.3,0.75) {x}; \node[] at (1.5,-0.3) {y}; \node[]
at (-0.3,1.5) {y}; \node[] at (2.25,-0.3) {z}; \node[] at
(-0.3,2.25) {z};
%\draw[step=1cm, gray, very thin](0,0) grid (3,3);
%\filldraw (1,1) circle[radius=1.5pt];
\end{tikzpicture}
%\hspace{2cm}
\caption{Case 1 and $x<y<z$ ('Fake' example)}\label{f1}
\end{figure}
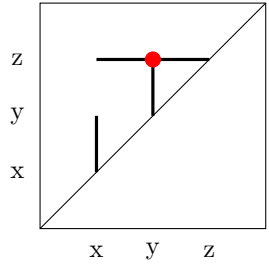

The red point signs the problem of this configuration, since two
lines with different values cross each other. There is no such a
quasitrivial monotone operation.

Thus this subcase provides 'fake' example to study associativity.
From the total, 8 cases are 'fake' in this sense.

The remaining 4 cases are presented in the statement. Figure
\ref{fig111} (a) and (b) represent the cases when equation
\eqref{eq1} holds, and Figure \ref{fig111} (c) and (d) represent the
cases when \eqref{eq2} holds.
\end{proof}

Since for a 2-element set none of the cases of Figure \ref{fig111}
can be realized, as an immediate consequence of Lemma \ref{lpic} we
get the following.
\begin{corollary}\label{cor2}
 Every quasitrivial nondecreasing operation $F:L_2^2\to L_2$ is associative.
\end{corollary}
%Another consequence arise
As a byproduct of this visualization we obtain a simple alternative proof for the following fact. This was proved first in \cite[Proposition 2]{Martin2003}. %For finite chains more can be stated see \cite[Proposition 11.]{Jimmy2017} (see Proposition \ref{prop:ane}). %We get the analogue result using visualization is the following.
\begin{corollary}
Let $X$ be nonempty chain and $F:X^2\to X$ be a quasitrivial
symmetric nondecreasing operation then $F$ is associative.
\end{corollary}
\begin{proof}
 If we add the assumption of symmetry of $F$, each cases presented in Figure \ref{fig111} have crossing lines (as in Figure \ref{figsym}), which is not possible. Thus $F$ is automatically associative.
%As we present first case of Figure \ref{fig111} in Figure \ref{figsym}. Since  %The equation $F(y,z)=F(z,y)=y$ and $z<x<y$ implies that $F(x,y)=y$. Thus Figure \ref{fig111} (a) is not possible for quasitrivial, nondecreasing operations. For other cases similar argument works. Therefore,
%none of the cases of Figure \ref{fig111} is possible, $F$ is automatically associative.
\begin{figure}[ht]
\begin{tikzpicture}
 %C/.style = {circle,thick,draw, inner sep=2pt},
\draw (0,0)-- (0,4)--(4,4)--(4,0) -- cycle; \draw (0,0)--(4,4);
\draw[very thick] (1,1)--(2,1); \draw[very thick] (2,2)--(2,3);
\draw[very thick] (3,3)--(3,1); \node[] at (1,-0.3) {z}; \node[] at
(-0.3,1) {z}; \node[] at (2,-0.3) {x}; \node[] at (-0.3,2) {x};
\node[] at (3,-0.3) {y}; \node[] at (-0.3,3) {y};
%\draw[step=1cm, gray, very thin](0,0) grid (4,4);
\node[] at (5.5, 2) {$\Longrightarrow$}; \draw (7,0)--
(7,4)--(11,4)--(11,0) -- cycle; \draw (7,0)--(11,4); \draw[very
thick] (8,1)--(9,1); \draw[very thick] (9,2)--(9,3); \draw[very
thick] (10,3)--(10,1); \draw[very thick, dashed] (8,1)--(8,2);
\draw[very thick, dashed] (9,2)--(10,2); \draw[very thick, dashed]
(10,3)--(8,3); \filldraw[red] (9,3) circle (0.1cm); \filldraw[red]
(10,2) circle (0.1cm); \node[] at (8,-0.3) {z}; \node[] at (6.7,1)
{z}; \node[] at (9,-0.3) {x}; \node[] at (6.7,2) {x}; \node[] at
(10,-0.3) {y}; \node[] at (6.7,3) {y};
%\draw[step=1cm, gray, very thin](7,0) grid (11,4);
%\node[] at (2,-1) {Case 1/b: $y<z<x$ };
%\filldraw (1,1) circle[radius=1.5pt];
\end{tikzpicture}
\caption{The symmetric case}\label{figsym}
\end{figure}
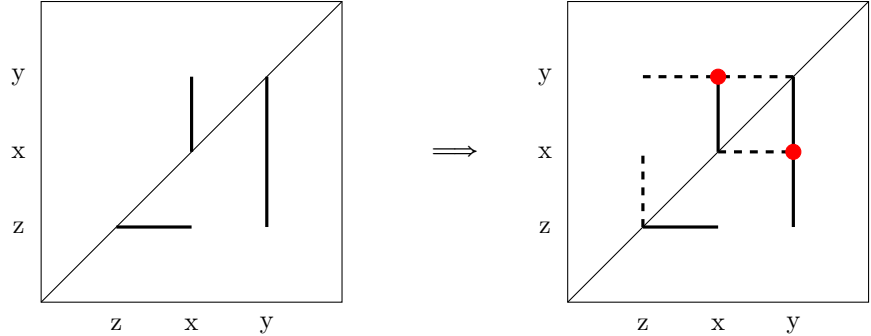
\end{proof}

For finite chains more can be stated.

\begin{proposition}[{\cite[Proposition 11.]{Jimmy2017}}]\label{prop:ane}
If $F:L_k^2\to L_k$ is quasitrivial symmetric nondecreasing then it
is associative and has a neutral element.
\end{proposition}
\begin{remark}
The conclusion that $F$ has a neutral element is not necessarily
true when $X=[0,1]$ (see \cite{Martin2003}).
This fact is one of the main difference between the cases $X=L_k$ and $X=[0,1]$.%(or $\mathbb{R}$).
\end{remark}

If we assume that $F$ has a neutral element (as it follows by
Proposition \ref{prop:ane} for finite chains), then as a consequence
of Observation \ref{l1} we get the following pictures (Figure
\ref{figsymne}) for quasitrivial monotone operations having neutral
elements. In Figure \ref{figsymne} the neutral element is denoted by
$e$.
\begin{figure}[ht]
\begin{tikzpicture}
 %C/.style = {circle,thick,draw, inner sep=2pt},
\draw (0,0)-- (0,4)--(4,4)--(4,0) -- cycle; \draw (0,0)--(4,4);
\draw[very thick] (0,0)--(1.5,0); \draw[very thick] (0,0)--(0,1.5);
\draw[very thick] (1,1)--(1.5,1); \draw[very thick] (1,1)--(1,1.5);
\draw[very thick] (2,2)--(1.5,2); \draw[very thick] (2,2)--(2,1.5);
\draw[very thick] (3,3)--(1.5,3); \draw[very thick] (3,3)--(3,1.5);
\draw[very thick] (4,4)--(1.5,4); \draw[very thick] (4,4)--(4,1.5);
\draw[very thick] (0.5,0.5)--(0.5,1.5); \draw[very thick]
(0.5,0.5)--(1.5,0.5); \draw[very thick] (2.5,2.5)--(2.5,1.5);
\draw[very thick] (2.5,2.5)--(1.5,2.5); \draw[very thick]
(3.5,3.5)--(3.5,1.5); \draw[very thick] (3.5,3.5)--(1.5,3.5);
\node[] at (1.5,-0.3) {e}; \node[] at (-0.3,1.5) {e}; \filldraw
(1.5,1.5) circle[radius=1.5pt];

\draw[loosely dotted] (5,1.5)-- (5,4)--(6.5,4); \draw[loosely
dotted] (9,1.5)-- (9,0)--(6.5,0); \draw
(6.5,1.5)--(9,1.5)--(9,4)--(6.5,4)--cycle; \draw
(6.5,1.5)--(5,1.5)--(5,0)--(6.5,0)--cycle; \node[] at (5.75,0.75)
{$x\wedge y$}; \node[] at (7.75,2.75) {$x\vee y$}; \node[] at
(6.5,-0.3) {e}; \node[] at (4.7,1.5) {e};
%\draw[step=1cm, gray, very thin](7,0) grid (11,4);
\end{tikzpicture}
\caption{Partial description of a quasitrivial monotone operations having neutral elements %symmetric idempotent monotone operation
}\label{figsymne}
\end{figure}
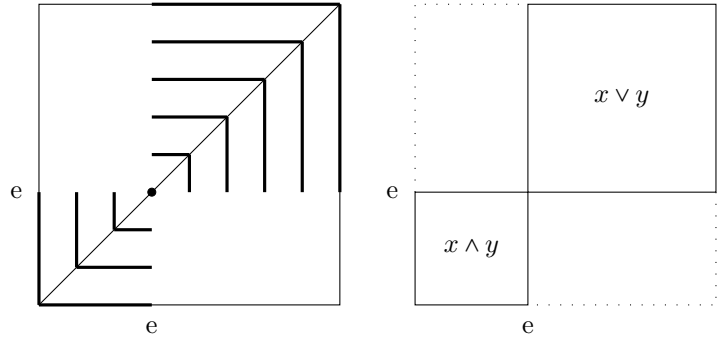

%It is important to note that this consequence that $F$ has a neutral element is one of the main differences between the discrete case $L_k$ and the continuous case $[0,1]$.

\section{Visual characterization of associative quasitrivial nondecreasing operations defined on $L_k$ }\label{s4}

From now on we
denote %$T=L_k^2$ and
the upper and the lower 'triangle' by
$$T_1=\{(x,y): x,y\in L_k, x\le y\},\ \ \ \ T_2=\{(x,y): x,y\in L_k, x\ge y\},$$
respectively, %and lower 'triangle' by
%$$T_2=\{(x,y): x,y\in L_k, x\le y\}.$$
as in Figure \ref{figtri}. We note that $T_1\cap T_2$ is the
diagonal $\Delta_{L_k}$.
\begin{figure}[ht]
\centering

\begin{tikzpicture}
 %C/.style = {circle,thick,draw, inner sep=2pt},
%\draw (0,0)-- (0,3)--(3,3)--(3,0) -- cycle;
%\filldraw[red]
%\draw (0,0)--(3,3);
\draw (0,0)-- (0,3)--(3,3)-- cycle; \draw[loosely dotted]
(3,3)--(3,0)--(0,0); \draw (7,3)--(7,0)--(4,0)-- cycle;
\draw[loosely dotted] (4,0)-- (4,3)--(7,3) ; \node[] at (1,2)
{$\huge{T_1}$}; \node[] at (6,1) {$\huge{T_2}$};
%\node[] at (-0.3,0.75) {y};
%\node[] at (1.5,-0.3) {x};
%\node[] at (-0.3,1.5) {x};
%\node[] at (2.25,-0.3) {z};
%\node[] at (-0.3,2.25) {z};
%\draw[step=1cm, gray, very thin](0,0) grid (3,3);

%\filldraw (1,1) circle[radius=1.5pt];
\end{tikzpicture}
%\hspace{2cm}

\caption{The upper and lower 'triangles' $T_1$ and
$T_2$}\label{figtri}
\end{figure}
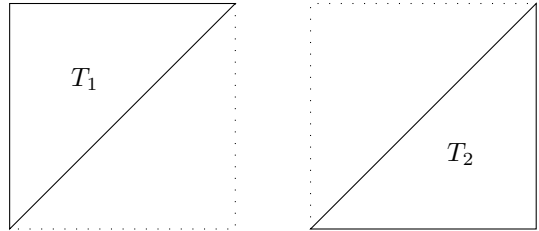
\begin{definition}
For a operation $F:L_k^2\to L_k$
%we call $F_1:L_k^2\to L_k$ and $F_2:L_k^2\to L_k$
there can be defined the {\it upper symmetrization $F_1$ and lower
symmetrization $F_2$} of $F$ as
\begin{equation*}
F_1(x,y)=\begin{cases}F(x,y) ~~~~&\textrm{ if } (x,y)\in T_1\\
F(y,x) ~~~~&\textrm{ if } (y,x)\in T_1
\end{cases}
~~~ \textrm{ and }~~~ F_2(x,y)=\begin{cases}F(x,y) ~~~~&\textrm{ if } (x,y)\in T_2\\
F(y,x) ~~~~&\textrm{ if } (y,x)\in T_2,
\end{cases}
\end{equation*}
Briefly, $F_1(x,y)=F(x\wedge y, x\vee y),\  F_2(x,y)=F(x\vee y,
x\wedge y)\ \  \forall x,y \in L_k$.
\end{definition}
%Since $F$ is idempotent the $F_1$ and $F_2 $
Fodor \cite{Fodor1996} (see also \cite[Theorem 2.6]{Sander}) shown
the following statement.
\begin{proposition}\label{prop:Fodor}
Let $X$ be a nonempty chain and $F:X^2\to X$ be an associative
operation. Then $F_1$ and $F_2$, the upper and the lower
symmetrization of $F$, are also associative.
\end{proposition}

This idea makes it possible to investigate the two 'parts' of a
non-symmetric associative operation as one-one half of two symmetric
associative operations.

By Proposition \ref{prop:ane}, both symmetrization of a
nondecreasing quasitrivial operation $F:L_k^2\to L_k$ has a neutral
element.
\begin{definition}
We call an element {\it upper (or lower) half-neutral element} of
$F$ if it is the neutral element of the upper (or the lower)
symmetrization. For simplicity we always denote the upper and lower
half-neutral element of $F$ by $e$ and $f$, respectively.
\end{definition}

Summarizing the previous results we get following partial
description.
\begin{proposition}\label{cor}
Let $F:L_k^2\to L_k$ be an associative quasitrivial nondecreasing
operation. Then it has an upper and an lower half-neutral element
denoted by $e$ and $f$. Moreover, if $e\le f$ then
%$F$ is a minimum if $< $
\begin{equation*}
F(x,y)=\begin{cases}
x\wedge y ~~~~&\textrm{ if } x\vee y\le e\\
y ~~~~&\textrm{ if } e \le x\le f\\
x \vee y ~~~~&\textrm{ if } f\le x\wedge y
\end{cases}
\end{equation*}
Analogously, if $f\le e$ then
%$F$ is a minimum if $< $
\begin{equation*}
F(x,y)=\begin{cases}
x\wedge y ~~~~&\textrm{ if } x\vee y\le f\\
x ~~~~&\textrm{ if } f \le x\le e\\
x \vee y ~~~~&\textrm{ if } e\le x\wedge y
\end{cases}
\end{equation*}
\end{proposition}

\begin{figure}[ht]
\begin{tikzpicture}
 %C/.style = {circle,thick,draw, inner sep=2pt},
\draw (0,0)-- (0,4)--(4,4)--(4,0) -- cycle; \draw (0,0)--(4,4);
\draw[very thick] (0,0)--(3,0); \draw[very thick] (0,0)--(0,1.5);
\draw[very thick] (1,1)--(3,1); \draw[very thick] (1,1)--(1,1.5);
\draw[very thick] (2,2)--(3,2); \draw[very thick] (2,2)--(1.5,2);
%\draw[very thick] (3,3)--(3,3);
\draw[very thick] (3,3)--(1.5,3); \draw[very thick] (4,4)--(4,3);
\draw[very thick] (4,4)--(1.5,4); \draw[very thick]
(0.5,0.5)--(0.5,1.5); \draw[very thick] (0.5,0.5)--(3,0.5);
\draw[very thick] (1.5,1.5)--(3,1.5); \draw[very thick]
(2.5,2.5)--(1.5,2.5); \draw[very thick] (2.5,2.5)--(3,2.5);
\draw[very thick] (3.5,3.5)--(1.5,3.5); \draw[very thick]
(3.5,3.5)--(3.5,3);
%\node[] at (1.5,-0.3) {e};
\node[] at (-0.3,1.5) {e}; \node[] at (3,-0.3) {f};
%\node[] at (-0.3,3.5) {f};
\filldraw (1.5,1.5) circle[radius=1.5pt]; \filldraw (3,3)
circle[radius=1.5pt];

\draw[loosely dotted] (5,1.5)-- (5,4)--(6.5,4); \draw[loosely
dotted] (9,3)-- (9,0)--(8,0); \draw (6.5,1.5)--(6.5,4); \draw
(8,0)--(8,3); \draw (6.5,0)--(8,0); \draw (6.5,4)--(8,4); \draw
(9,4)--(9,3)--(8,3)--(8,4)--cycle; \draw
(6.5,1.5)--(5,1.5)--(5,0)--(6.5,0)--cycle; \node[] at (5.75,0.75)
{$x\wedge y$};
\node[] at (8.5,3.5) {$x\vee y$ };%\tiny{$\max(x,y)$}};
\node[] at (7.25,2) {Proj$_y$};

\node[] at (4.7,1.5) {e}; \node[] at (8,-0.3) {f};
%\node[] at (4.7,1.5) {e};
%\draw[step=1cm, gray, very thin](7,0) grid (11,4);
\end{tikzpicture}
\caption{Partial description of associative quasitrivial monotone
operations when $e\le f$}\label{figprtde}
\end{figure}

We note that $e=f$ iff $F$ has a neutral element.

%Proposition 4.4 is illustrated in Figure \ref{figprtde} when $e\le f$.
%\newpage
The following lemma is essential for the visual  characterization.

\begin{lemma}\label{limp}
Let $F:L_k^2\to L_k$ be an associative quasitrivial nondecreasing
operation. Assume that there exists $a<b\in L_k$ such that
$F(a,b)=a$ and $F(b,a)=b$. Then one of the following holds:
\begin{enumerate}
\item[(a)] If $F(a+1, a)=a$, then
\begin{equation*}
F(x,b)=b \textrm{ and } F(y, a)=a
\end{equation*}
for every $x\in[a+1,b]$ and $y\in [a,b-1]$.
\item[(b)] If $F(a+1,a)=a+1$, then
$F(x,y)=x \ (=Proj_x)$ for all $x,y\in [a,b]$. % Moreover, $a$ and $b$ is the lower and the upper half-neutral element of $F$, respectively.
\begin{figure}[ht]
\begin{tikzpicture}
 %C/.style = {circle,thick,draw, inner sep=2pt},
%\draw (0,0)-- (0,3)--(3,3)--(3,0) -- cycle;
%\filldraw[red]
%\draw (0,0)--(3,3);
%\node[] at (2,7.3) {};

\draw (2.5,4)-- (2.5,7); \draw (5.5,4)-- (5.5,7);
%\scalebox{0.4};
\draw[loosely dotted] (2.5,7)--(5.5,7)--(5.5,4)--(2.5,4)--cycle;
\node[] at (2.5,3.7) {a};
%\node[] at (1.5,-0.3) {a};
\node[] at (5.5,3.7) {b};
%\node[] at (1.5,-0.3) {b};
\node[] at (4,3.5) {$\huge{\Downarrow}$}; \draw (0,0)-- (0,3); \draw
(3,0)-- (3,3); \draw (2.6,0)--(0,0); \draw (3,3)--(0.4, 3); \node[]
at (0,-0.3) {a}; \node[] at (0.6,-0.3) {a+1}; \node[] at (3,-0.3)
{b}; \node[] at (2.6,-0.3) {b-1};
%\draw[loosely dotted]
%(0,3)--(3,3)--(3,0)--(0,0)--cycle;

\node[] at (4,1.5) { or };

%\draw (8,0)-- (8,3);
\draw (5,0)-- (5,3);
%\draw (5.2,0)-- (5.2,3);
\draw (5.4,0)-- (5.4,3);
%\draw (5.6,0)-- (5.6,3);
\draw (5.8,0)-- (5.8,3);
%\draw (6,0)-- (6,3);
\draw (6.2,0)-- (6.2,3);
%\draw (6.4,0)-- (6.4,3);
\draw (6.6,0)-- (6.6,3);
%\draw (6.8,0)-- (6.8,3);
\draw (7,0)-- (7,3);
%\draw (7.2,0)-- (7.2,3);
\draw (7.4,0)-- (7.4,3);
%\draw (7.6,0)-- (7.6,3);
\draw (7.8,0)-- (7.8,3); \node[] at (5,-0.3) {a}; \node[] at
(7.8,-0.3) {b};
%\draw[loosely dotted]
%(0,3)--(3,3)--(3,0)--(0,0)--cycle;\node[] at (-0.3,0.75) {y};
%\node[] at (1.5,-0.3) {x};
%\node[] at (-0.3,1.5) {x};
%\node[] at (2.25,-0.3) {z};
%\node[] at (-0.3,2.25) {z};
%\draw[step=1cm, gray, very thin](0,0) grid (3,3);

%\filldraw (1,1) circle[radius=1.5pt];
\end{tikzpicture}
\caption{Graphical interpretation of Lemma
\ref{limp}}\label{figlimpa}
\end{figure}
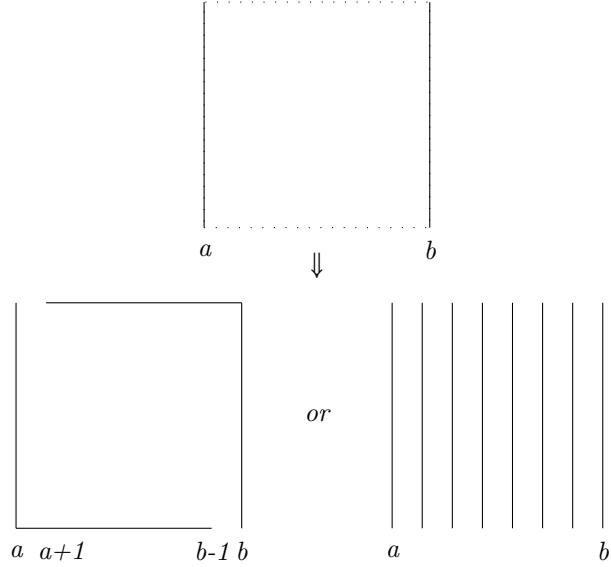
%\hspace{2cm}
\end{enumerate}
\end{lemma}

\begin{proof}
Assume first that $F(a+1,a)=a$. Then it follows that $F(a+1,b)=b$,
otherwise we get Figure \ref{fig111} (a). Using Observation \ref{l1}
we have that $F(x,b)=b$ for every $x\in [a+1,b]$. The equation
$F(b-1,b)=b$ implies that $F(b-1, a)=a$, otherwise  we are in the
situation of Figure \ref{fig111} (b). Similarly, as above we get
that $F(y,a)=a$ for every $y\in[a,b-1]$. Here we note that an
analogue argument gives the same result if we assume originally that
$F(b-1,b)=b$.

Now assume that $F(a+1,a)=a+1$. This immediately implies that
$F(x,a)=x$ for every $x\in[a,b]$ by quasitriviality, since it cannot
be $a$ by the nondecreasingness of $F$. Using Observation \ref{l1}
again, it follows that $F(x,y)=x$ for all $y\in [a,x]$. Since
$F(b-1,b)=b$ also implies the previous case, the assumption
$F(a+1,a)=a+1$ implies $F(b-1,b)=b-1$. Similarly as above, this
condition implies that $F(x,b)=x$ for all $x\in[a,b]$ and, by
Observation \ref{l1}, it follows that $F(x,y)=x$ for every $y\in
[x,b]$. Altogether we get that $F(x,y)=x=\textrm{Proj}_x(x,y)$ as we
stated.
\end{proof}

\begin{remark}\label{rimp}
Analogue of Lemma \ref{limp} can be formalized as follows.

{\it Let $F:L_k^2\to L_k$ be an associative quasitrivial
nondecreasing operation. Assume that there exists $a<b\in L_k$ such
that  $F(b,a)=a$ and $F(a,b)=b$. Then one of the following holds:
\begin{enumerate}
\item[(a)] If $F(a,a+1)=a$, then
\begin{equation*}
F(b,x)=b \textrm{ and } F(a,y)=a
\end{equation*}
for every $x\in[a+1,b]$ and $y\in [a,b-1]$.
\item[(b)] If $F(a,a+1)=a+1$, then
$F(x,y)=y(=Proj_y)$ for all $x,y\in [a,b]$.% Moreover, $a$ and $b$ is the upper and the lower half-neutral element of $F$, respectively.
\begin{figure}[ht]
\begin{tikzpicture}
 %C/.style = {circle,thick,draw, inner sep=2pt},
%\draw (0,0)-- (0,3)--(3,3)--(3,0) -- cycle;
%\filldraw[red]
%\draw (0,0)--(3,3);
%\node[] at (2,7.3) {};

\draw (2.5,4)-- (5.5,4); \draw (2.5,7)-- (5.5,7);
%\scalebox{0.4};
\draw[loosely dotted] (2.5,7)--(5.5,7)--(5.5,4)--(2.5,4)--cycle;
\node[] at (2.2,4) {a};
%\node[] at (1.5,-0.3) {a};
\node[] at (2.2,7) {b};
%\node[] at (1.5,-0.3) {b};
\node[] at (4,3.5) {$\huge{\Downarrow}$}; \draw (0,0)-- (3,0); \draw
(0,3)-- (3,3); \draw (0,2.6)--(0,0); \draw (3,3)--(3,0.4); \node[]
at (-0.3,0) {a}; \node[] at (-0.4,0.4) {a+1}; \node[] at (-0.3,3)
{b}; \node[] at (-0.4,2.6) {b-1};
%\draw[loosely dotted]
%(0,3)--(3,3)--(3,0)--(0,0)--cycle;

\node[] at (4,1.5) { or };

%\draw (8,0)-- (8,3);
\draw (5,0)-- (8,0);
%\draw (5.2,0)-- (5.2,3);
\draw (5,0.4)-- (8,0.4);
%\draw (5.6,0)-- (5.6,3);
\draw (5,0.8)-- (8,0.8);
%\draw (6,0)-- (6,3);
\draw (5,1.2)-- (8,1.2);
%\draw (6.4,0)-- (6.4,3);
\draw (5,1.6)-- (8,1.6);
%\draw (6.8,0)-- (6.8,3);
\draw (5,2)-- (8,2);
%\draw (7.2,0)-- (7.2,3);
\draw (5,2.4)-- (8,2.4);
%\draw (7.6,0)-- (7.6,3);
\draw (5,2.8)-- (8,2.8); \node[] at (4.7,0) {a}; \node[] at (4.7,3)
{b};
%\draw[loosely dotted]
%(0,3)--(3,3)--(3,0)--(0,0)--cycle;\node[] at (-0.3,0.75) {y};
%\node[] at (1.5,-0.3) {x};
%\node[] at (-0.3,1.5) {x};
%\node[] at (2.25,-0.3) {z};
%\node[] at (-0.3,2.25) {z};
%\draw[step=1cm, gray, very thin](0,0) grid (3,3);

%\filldraw (1,1) circle[radius=1.5pt];
\end{tikzpicture}
\caption{Graphical interpretation of Remark
\ref{rimp}}\label{figrimpa}
\end{figure}
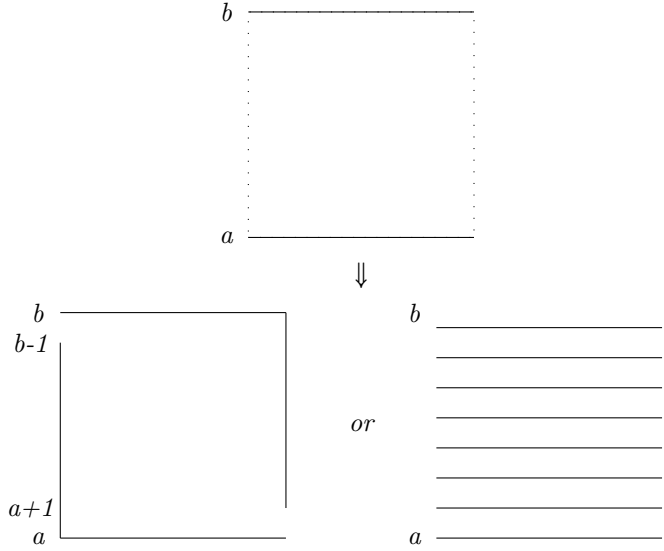
%\hspace{2cm}
\end{enumerate}
} The proof of this statement is analogue to Lemma \ref{limp} using
Figure \ref{fig111}(c) and (d) instead of Figure \ref{fig111}(a) and
(b), respectively.
\end{remark}

%\bigskip
%Structure:

%Lemma: partial description of $F$

%Lemma: Reduction process starting from the 'corners'. (3 cases)

%Theorem: Characterization .
\newpage
From the previous results we conclude the following.
\begin{lemma}\label{labin} Let $F:L_k^2\to L_k$ be an associative quasitrivial and nondecreasing operation and $e$ and $f$ the upper and the lower half-neutral elements, respectively, and let $a,b\in L_k$ ($a<b$) be given. %such that $a<b$.
If  $F(x,y)=x$ for every $x,y\in [a,b]$ (i.e, Lemma \ref{limp} (b)
holds), then $f< e$ and $[a,b]\subseteq [f,e]$. Similarly, if
$F(x,y)=y$ for every $x,y\in [a,b]$ (i.e, Remark \ref{rimp} (b)
holds), then $e< f$ and $[a,b]\subseteq [e,f]$.
\end{lemma}
\begin{proof}
This is a direct consequence of Proposition \ref{cor}. If $a$ or $b$
is not in $[e\wedge f,e\vee f]$ then $\tilde{F}=F|_{[a,b]^2}$
contains a part where $\tilde{F}$ is a minimum or a maximum.
Moreover, it is also easily follows that if $F(x,y)=x$ for every
$x,y\in [a,b]$, then $f< e$ must hold. Similarly, $F(x,y)=y$ for
every $x,y\in [a,b]$ implies $e< f$.
\end{proof}
\begin{corollary}\label{corcases}
 Let $F, e, f$ be as in Lemma \ref{labin} and assume that $a,b\in X$ such that $a<b$ and $F(a,b)\ne F(b,a)$. Then
 \begin{enumerate}
 \item[(i)] Lemma \ref{limp}(b) holds iff $f< e$ and $a,b\in [f,e]$,
 \item[(ii)] Remark \ref{rimp}(b) holds iff $e< f$ and $a,b\in [e,f]$.
 \item[(iii)] Lemma \ref{limp}(a) or  Remark \ref{rimp}(a) holds iff $a,b\not\in[e\wedge f,e\vee f]$.
 \end{enumerate}
 \end{corollary}

With other words we have:
\begin{corollary}
 Let $F, e, f$ be as in Lemma \ref{labin}. Then $F(a,b)=F(b,a)$, if $a\not\in [e\wedge f,e\vee f]$ and $b\in [e\wedge f,e\vee f]$, or $b\not\in [e\wedge f,e\vee f]$ and $a\in [e\wedge f,e\vee f]$.
 \end{corollary}

This form makes it possible to extend the partial description. (See
Figure \ref{figeprtde} for the case $e<f$.)

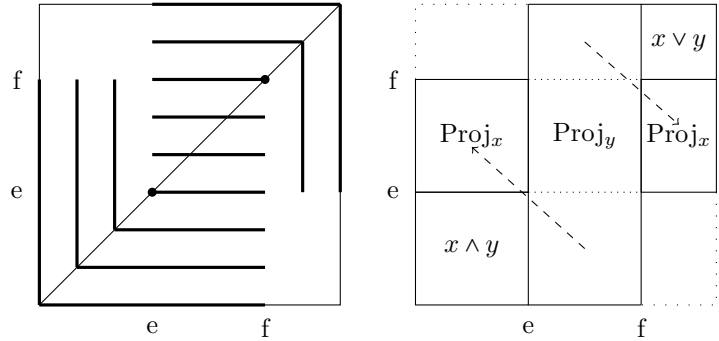
\begin{figure}[ht]
\begin{tikzpicture}
 %C/.style = {circle,thick,draw, inner sep=2pt},
\draw (0,0)-- (0,4)--(4,4)--(4,0) -- cycle; \draw (0,0)--(4,4);
\draw[very thick] (0,0)--(3,0); \draw[very thick] (0,0)--(0,3);
\draw[very thick] (1,1)--(3,1); \draw[very thick] (1,1)--(1,3);
\draw[very thick] (2,2)--(3,2); \draw[very thick] (2,2)--(1.5,2);
%\draw[very thick] (3,3)--(3,3);
\draw[very thick] (3,3)--(1.5,3); \draw[very thick] (4,4)--(4,1.5);
\draw[very thick] (4,4)--(1.5,4); \draw[very thick]
(0.5,0.5)--(0.5,3); \draw[very thick] (0.5,0.5)--(3,0.5); \draw[very
thick] (1.5,1.5)--(3,1.5); \draw[very thick] (2.5,2.5)--(1.5,2.5);
\draw[very thick] (2.5,2.5)--(3,2.5); \draw[very thick]
(3.5,3.5)--(1.5,3.5); \draw[very thick] (3.5,3.5)--(3.5,1.5);
\node[] at (1.5,-0.3) {e}; \node[] at (-0.3,1.5) {e}; \node[] at
(3,-0.3) {f}; \node[] at (-0.3,3) {f}; \filldraw (1.5,1.5)
circle[radius=1.5pt]; \filldraw (3,3) circle[radius=1.5pt];

\draw[loosely dotted] (5,3)-- (5,4)--(6.5,4); \draw[loosely dotted]
(9,1.5)-- (9,0)--(8,0); \draw[dotted] (6.5,3)-- (8,3); \draw
[dashed,->] (7.25,0.75)--(5.75, 2.1); \draw[dotted] (6.5,1.5)--
(8,1.5); \draw [dashed,->] (7.25,3.5)--(8.5, 2.4);

\draw (6.5,1.5)--(6.5,4); \draw (8,0)--(8,3); \draw (6.5,0)--(8,0);
\draw (6.5,4)--(8,4); \draw
(5,1.5)--(5,3)--(6.5,3)--(6.5,1.5)--cycle; \draw
(8,1.5)--(8,3)--(9,3)--(9,1.5)--cycle; \draw
(9,4)--(9,3)--(8,3)--(8,4)--cycle; \draw
(6.5,1.5)--(5,1.5)--(5,0)--(6.5,0)--cycle; \node[] at (5.75,0.75)
{$x\wedge y$};
\node[] at (8.5,3.5) {$x\vee y$ };%\tiny{$\max(x,y)$}};
\node[] at (7.25,2.25) {Proj$_y$}; \node[] at (5.75,2.25)
{Proj$_x$}; \node[] at (8.5,2.25) {Proj$_x$};

\node[] at (4.7,1.5) {e}; \node[] at (6.5,-0.3) {e}; \node[] at
(8,-0.3) {f}; \node[] at (4.7,3) {f};
%\draw[step=1cm, gray, very thin](7,0) grid (11,4);
\end{tikzpicture}
\caption{Extended partial description of associative quasitrivial
monotone operations when $e< f$}\label{figeprtde}
\end{figure}

Using Lemma \ref{limp} and Remark \ref{rimp} we can provide a visual
characterization of associative quasitrivial nondecreasing
operations. The characterization based on the following algorithm
which outputs the contour plot of $F$.
%The algorithm works as follows. In each step we choose

%Let be given an associative quasitrivial nondecreasing operation $F:L_k^2\to L_k$ and $Q_i$ denote the square what we have at the beginning of the $i^{th}$ step. In particular $Q_1$ denotes the square $L_n^2$, the domain of $F$. In each step of the algorithm we repeat essentially the same step.
Before we present the algorithm we note that the letters indicated
in the following figures represent the value of operation $F$ in the
corresponding points or lines (not a coordinate of the points itself
as usual).

{\bf Algorithm}
\begin{enumerate}
\item[Initial setting:] Let $Q_1=L_k^2$ and  $F:L_k^2\to L_k$ be an associative quasitrivial nondecreasing operation.
\item[Step i.] For $Q_i=[a,b]^2$ ($a\le b$) we distinguish cases according to the values of $F(a,b)$ and $F(b,a)$. Whenever $Q_i$ contains only 1 element ($a=b$) for some $i$, then we are done.
%\begin{enumerate}
\item[I. (a)]% The symmetric case
If $F(a,b)=F(b,a)=a$, then draw  straight lines between the points
$(b,a)$ and $(a,a)$ and between $(a,b)$ and $(a,a)$. Let
$Q_{i+1}=[a+1,b]^2$. (See Figure \ref{figmove1}.)
\begin{figure}[ht]
\begin{tikzpicture}

\node[] at (3,-0.3) {a}; \filldraw (3,0) circle[radius=1.5pt];
\node[] at (-0.3,3) {a}; \filldraw (0,3) circle[radius=1.5pt];

\draw[loosely dotted] (0,3)--(3,3)--(3,0)--(0,0)--cycle;

\node[] at (4,1.5) { $\Longrightarrow$ };

\draw (5.4,0.4)--(5.4,3)--(8,3)--(8, 0.4)--cycle; \node[] at
(6.5,-0.3) {a}; \node[] at (4.7,1.5) {a}; \draw (5,3)--(5,0)--(8,0);

\node[] at (6.7,1.7) {$Q_{i+1}$};
\end{tikzpicture}
\caption{Case I.(a)}\label{figmove1}
\end{figure}
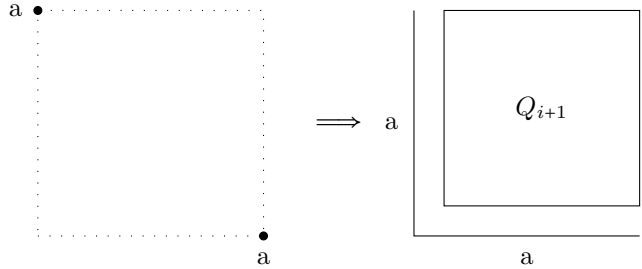
%\begin{enumerate}
%\item[(a)]
 \item[(b)] If $F(a,b)=F(b,a)=b$, then draw straight lines between the points $(a,b)$ and $(b,b)$ and between $(b,a)$ and $(b,b)$. Let $Q_{i+1}=[a,b-1]^2$.
 %\end{enumerate}

 \item[II. (a)]% Projections

 If $F(a,b)=a, F(b,a)=b$ and $F(a+1, a)=a+1$, then $F(x,y)=x$ for all $x,y\in [a,b]$ and we are done. (See Figure \ref{figmove2})
  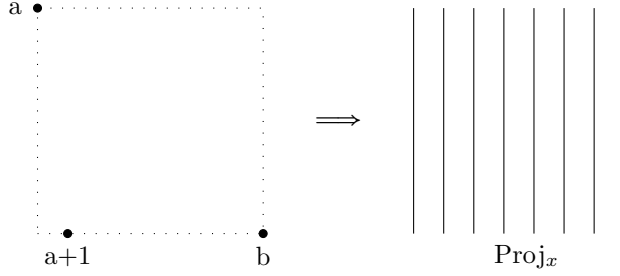
\begin{figure}[ht]
\begin{tikzpicture}

\node[] at (3,-0.3) {b}; \filldraw (3,0) circle[radius=1.5pt];
\node[] at (-0.3,3) {a}; \filldraw (0,3) circle[radius=1.5pt];
\node[] at (0.4,-0.3) {a+1}; \filldraw (0.4,0) circle[radius=1.5pt];

\draw[loosely dotted] (0,3)--(3,3)--(3,0)--(0,0)--cycle;

\node[] at (4,1.5) { $\Longrightarrow$ };

%\node[] at (6.5,3.3) {b};
%\node[] at (8.15,1.5) {b};
%\draw (5,3)--(7.8,3)--(7.8,0);

\draw (5,0)-- (5,3);
%\draw (5.2,0)-- (5.2,3);
\draw (5.4,0)-- (5.4,3);
%\draw (5.6,0)-- (5.6,3);
\draw (5.8,0)-- (5.8,3);
%\draw (6,0)-- (6,3);
\draw (6.2,0)-- (6.2,3);
%\draw (6.4,0)-- (6.4,3);
\draw (6.6,0)-- (6.6,3);
%\draw (6.8,0)-- (6.8,3);
\draw (7,0)-- (7,3);
%\draw (7.2,0)-- (7.2,3);
\draw (7.4,0)-- (7.4,3);
%\draw (7.6,0)-- (7.6,3);
\draw (7.8,0)-- (7.8,3); \node[] at (6.5,-0.3) {Proj$_x$};
%\node[] at (7.8,-0.3) {b};
%\node[] at (6.7,1.7) {$Q_{i+1}$};
\end{tikzpicture}
\caption{Case II.(a)}\label{figmove2}
\end{figure}

\item[(b)]
 If $F(a,b)=b, F(b,a)=a$ and $F(a,a+1)=a+1$, then  $F(x,y)=y$ for all $x,y\in [a,b]$ and we are also done.
%\end{enumerate}

 \item[III. (a)]
  If $F(a,b)=a, F(b,a)=b$ and $F(a+1, a)=a$, then Lemma \ref{limp} (a) holds and we have Figure \ref{figmove3}. Let $Q_{i+1}=[a+1,b-1]^2$.
  \begin{figure}[ht!]
\begin{tikzpicture}
\node[] at (3.3,0) {b}; \filldraw (3,0) circle[radius=1.5pt];
\node[] at (-0.3,3) {a}; \filldraw (0,3) circle[radius=1.5pt];
\node[] at (0.4,-0.3) {a}; \filldraw (0.4,0) circle[radius=1.5pt];

\draw[loosely dotted] (0,3)--(3,3)--(3,0)--(0,0)--cycle;

\node[] at (4,1.5) { $\Longrightarrow$ }; \draw (5,0)-- (5,3); \draw
(8,0)-- (8,3); \draw (7.6,0)--(5,0); \draw (8,3)--(5.4, 3); \node[]
at (6.5,-0.3) {a};
%\filldraw (6.5,-0.3) circle[radius=1.5pt];
\node[] at (4.7,1.5) {a};
%\filldraw (4.7,1.5) circle[radius=1.5pt];
\node[] at (6.5,3.3) {b}; \node[] at (8.15,1.5) {b};

\draw (5.4,0.4)--(5.4,2.6)--(7.6,2.6)--(7.6,0.4)--cycle; \node[] at
(6.5,1.5) {$Q_{i+1}$};
\end{tikzpicture}
\caption{Case III.(a)}\label{figmove3}
\end{figure}
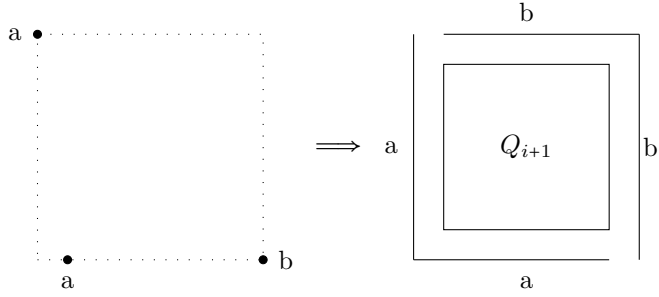
%\end{enumerate}

\item[(b)]
 If $F(a,b)=b, F(b,a)=a$ and $F(a,a+1)=a$, then  Remark \ref{rimp} (a) holds.
 Let $Q_{i+1}=[a+1,b-1]^2$.
\end{enumerate}

%\end{enumerate}
%\end{enumerate}
It is clear that the algorithm is finished after finitely many
steps. Let us denote this number of steps by $l\in \mathbb{N}$.

We also denote the top-left and the bottom-right corner of $Q_i$ by
$p_i$ and $q_i$ ($i=1,\dots,l$), respectively.

Let $\mathcal{P}$ (and $\mathcal{Q}$) denote the path containing
$p_i$ (and $q_i$) for $i\in \{1,\dots, l\}$ and line segments
between consecutive $p_i$'s (and $q_i$'s). Let us denote the line
segment between $p_i$ and $p_{i+1}$ by $\overline{p_i,p_{i+1}}$. We
set the notation $\mathcal{P}=(p_j)_{j=1}^l$ and
$\mathcal{Q}=(q_j)_{j=1}^l$.

Clearly, we get the path $\mathcal{P}$ if we start at  the top-left
corner of $L_k^2$ and in each step we move either one place to the
right or one  place downward or one place diagonally downward-right.

\begin{definition}
We say that a path is a {\it downward-right path} of $L_k$ if in
each step it moves to the nearest point of $L_k^2$ either one place
to the right or one place downward or one place diagonally
downward-right.

 \begin{figure}[ht!]
\begin{tikzpicture}
%\node[] at (3,-0.3) {b};
\filldraw (0,3) circle[radius=1.5pt];

\filldraw (0.3,2.7) circle[radius=1.5pt]; \filldraw (0.6,2.7)
circle[radius=1.5pt]; \node[] at (0.6,2.4) {$\mathcal{P}$};

\filldraw (0.9,2.7) circle[radius=1.5pt]; \filldraw (1.2,2.4)
circle[radius=1.5pt]; \filldraw[red] [->] (0.9,2.7) -- (0.9,2.4);
\filldraw[red] [->] (0.9,2.7) -- (1.2,2.7); \filldraw[red] [->]
(0.9,2.7) -- (1.15,2.45);

\filldraw(1.2,2.1) circle[radius=1.5pt];

%\node[] at (-0.3,3) {a};
%\filldraw (0,3) circle[radius=1.5pt];
%\node[] at (0.4,-0.3) {a};
%\filldraw (0.4,0) circle[radius=1.5pt];
\draw[loosely dotted] (0,3)--(3,3)--(3,0)--(0,0)--cycle;
\draw[dotted]
(0,3)--(0.3,2.7)--(0.6,2.7)--(0.9,2.7)--(1.2,2.4)--(1.2,2.1);
\draw[loosely dotted]
(1.2,2.1)--(1.2,1.2)--(2.1,1.2)--(2.1,2.1)--cycle; \filldraw (3,0)
circle[radius=1.5pt]; \filldraw (2.7,0.3) circle[radius=1.5pt];
\filldraw (2.7,0.6) circle[radius=1.5pt]; \node[] at (2.4,0.6)
{$\mathcal{Q}$}; \filldraw (2.7,0.9) circle[radius=1.5pt]; \filldraw
(2.4,1.2) circle[radius=1.5pt]; \filldraw (2.1,1.2)
circle[radius=1.5pt]; \draw[dotted](3,0)--(2.7,0.3) -- (2.7,0.6)--
 (2.7,0.9)-- (2.4,1.2) -- (2.1,1.2);
\end{tikzpicture}
\caption{The path $\mathcal{P}$ is a downward-right path}
\end{figure}
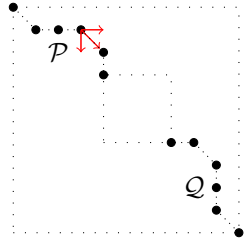

\end{definition}

If  $\overline{p_i, p_{i+1}}$ is horizontal or vertical, then the reduction from $Q_i$ to $Q_{i+1}$ is uniquely determined. Moreover, if $\overline{p_i, p_{i+1}}$ is horizontal, then $F(x,y)=F(y,x)=x\wedge y$, where $p_i=(x,y)$ and $q_i=(y,x)$. Similarly, if $\overline{p_i, p_{i+1}}$ is vertical, then $F(x,y)=F(y,x)=x\vee y$, where $p_i=(x,y)$ and $q_i=(y,x)$. %See Figure...
On the other hand if $\overline{p_i, p_{i+1}}$ is diagonal, then we
have a free choice for the value of $F$ in $p_i$. This is determined
by either Lemma \ref{limp} (a) or Remark \ref{rimp} (a). Since in
this case the value of $F$ in $q_i$ is different from $p_i$, the
value in $q_i$ is automatically defined. It is also clear from the
algorithm that the path $\mathcal{Q}$ is the reflection of
$\mathcal{P}$ to the diagonal $\Delta_{L_k}$.

Using the previous paragraph and Observation \ref{l1} it is possible
to reconstruct operations from a given downward-right path
$\mathcal{P}$ which starts at $p_1=(1,k)$.
\begin{example} We illustrate the reconstruction on $L_6\times L_6$. The paths $\mathcal{P}=(p_j)_{j=1}^5$ and  $\mathcal{Q}=(q_j)_{j=1}^5$ denoted by red and blue, respectively. According to the previous observations we get the following pictures (see Figure \ref{figrec}). It can be clearly seen that $\mathcal{Q}$ is the reflection of $\mathcal{P}$ to the diagonal $\Delta_{L_6}$, and $4$ is the neutral element of the reconstructing operation, where $\mathcal{P}$ and $\mathcal{Q}$ touch each other and reach the diagonal $\Delta_{L_6}$. For the precise statement and proof see Theorem \ref{thmchar}.

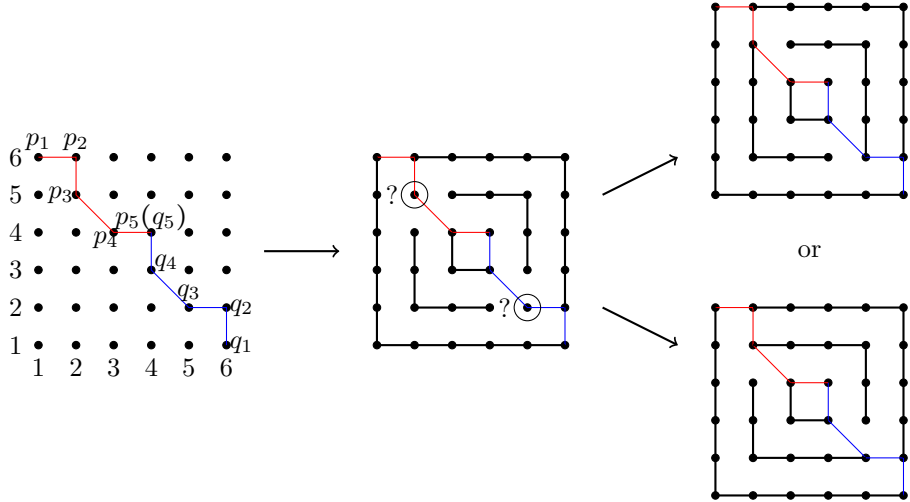
\begin{figure}[ht!]
\begin{tikzpicture}

% \draw[style=help lines,thick] (0,0) grid[step=.5cm] (2.5,2.5);

    \foreach \x in {0,1,...,5}
    {
        \foreach \y in {0,1,...,5}
        {
            \node[draw,circle,inner sep=1pt,fill] at (.5*\x,.5*\y) {};
        }
    }

\node[] at (0,-0.3) {1}; \node[] at (0.5,-0.3) {2}; \node[] at
(1,-0.3) {3}; \node[] at (1.5,-0.3) {4}; \node[] at (2,-0.3) {5};
\node[] at (2.5,-0.3) {6};
%\filldraw (3,0) circle[radius=1.5pt];
\node[] at (-0.3,0) {1}; \node[] at (-0.3,0.5) {2}; \node[] at
(-0.3,1) {3}; \node[] at (-0.3,1.5) {4}; \node[] at (-0.3,2) {5};
\node[] at (-0.3,2.5) {6};

%\draw (0,3) circle[radius=1.5pt];
\draw[red] (0,2.5)--(0.5,2.5)--(0.5,2)--(1,1.5)--(1.5,1.5);
\draw[blue] (1.5,1.5)--(1.5,1)--(2,0.5)--(2.5,0.5)--(2.5,0);
%\draw[loosely dotted] (0,3)--(3,3)--(3,0)--(0,0)--cycle;

\node[] at (0,2.7) {$p_1$}; \node[] at (0.5,2.7) {$p_2$}; \node[] at
(0.3,2) {$p_3$}; \node[] at (0.9,1.4) {$p_4$}; \node[] at (1.5,1.7)
{$p_5(q_5)$};
%\node[] at (0,2.7) {$p_1$};

\node[] at (2.7,0) {$q_1$}; \node[] at (2.7, 0.5) {$q_2$}; \node[]
at (2, 0.7) {$q_3$}; \node[] at (1.7,1.1) {$q_4$};
%\node[] at (0,2.7) {$q_5$};
%\node[] at (0,2.7) {$p_1$};

%\node[] at (3.75,1.25) { $\Longrightarrow$ };

%\draw[->, thick] ();

    \foreach \x in {0,1,...,5}
    {
        \foreach \y in {0,1,...,5}
        {
            \node[draw,circle,inner sep=1pt,fill] at (.5*\x+4.5,.5*\y) {};
        }
    }

\draw[red] (4.5,2.5)--(5,2.5)--(5,2)--(5.5,1.5)--(6,1.5);
\draw[blue](6,1.5)--(6,1)--(6.5,0.5)--(7,0.5)--(7,0);

\draw[thick] (4.5,2.5)--(4.5,0)--(7,0); \draw[thick]
(5,2.5)--(7,2.5)--(7,0.5); \draw[thick] (5,1.5)--(5,0.5)--(6,0.5);
\draw[thick] (5.5,2)--(6.5,2)--(6.5,1); \draw[thick]
(5.5,1.5)--(5.5,1)--(6,1);

\draw[->,thick]  (3,1.25)--(4,1.25); \draw[->, thick] (7.5,2)--(8.5,
2.5); \draw[->, thick] (7.5,0.5)--(8.5, 0); \node[] at (10.25,1.25)
{ or };
 \draw (5,2) circle[radius= 0.5 em];
 %\draw (5,2) circle[radius=1.5pt];
\draw (6.5,0.5) circle[radius= 0.5 em]; \node[] at (4.7,2) {?};
\node[] at (6.2, 0.5) {?};

    \foreach \x in {0,1,...,5}
    {
        \foreach \y in {0,1,...,5}
        {
            \node[draw,circle,inner sep=1pt,fill] at (.5*\x+9,.5*\y+2) {};
        }
    }

  \foreach \x in {0,1,...,5}
    {
        \foreach \y in {0,1,...,5}
        {
            \node[draw,circle,inner sep=1pt,fill] at (.5*\x+9,.5*\y-2) {};
        }
    }

\draw[red] (9,4.5)--(9.5,4.5)--(9.5,4)--(10,3.5)--(10.5,3.5);
\draw[blue](10.5,3.5)--(10.5,3)--(11,2.5)--(11.5,2.5)--(11.5,2);

\draw[thick] (9,4.5)--(9,2)--(11.5,2); \draw[thick]
(9.5,4.5)--(11.5,4.5)--(11.5,2.5); \draw[thick]
(9.5,4)--(9.5,2.5)--(10.5,2.5); \draw[thick]
(10,4)--(11,4)--(11,2.5); \draw[thick] (10,3.5)--(10,3)--(10.5,3);

\draw[red] (9,0.5)--(9.5,0.5)--(9.5,0)--(10,-0.5)--(10.5,-0.5);
\draw[blue](10.5,-0.5)--(10.5,-1)--(11,-1.5)--(11.5,-1.5)--(11.5,-2);

\draw[thick] (9,0.5)--(9,-2)--(11.5,-2); \draw[thick]
(9.5,0.5)--(11.5,0.5)--(11.5,-1.5); \draw[thick]
(9.5,-0.5)--(9.5,-1.5)--(11,-1.5); \draw[thick]
(9.5,0)--(11,0)--(11,-1); \draw[thick]
(10,-0.5)--(10,-1)--(10.5,-1);
%\draw (5.4,0.4)--(5.4,3)--(8,3)--(8, 0.4)--cycle;
%\node[] at (6.5,-0.3) {a};
%\node[] at (4.7,1.5) {a};
%\draw (5,3)--(5,0)--(8,0);

%\node[] at (6.2,1.7) {$Q_{i+1}$};
\end{tikzpicture}
\caption{Reconstruction of $F$ from the path
$\mathcal{P}$}\label{figrec}
\end{figure}

\end{example}

\begin{definition} Let $\mathcal{P}\subset L_k^2$ be the downward-right path from $(1,k)$ to $(a,b)$ ($a<b$) and let $\mathcal{Q}$ be  the reflection of $\mathcal{P}$ to the diagonal $\Delta_{L_k}$.

We say that $(x,y)\in L_k^2\setminus(\mathcal{P} \cup \mathcal{Q}
\cup [a,b]^2)$ is {\it above} $\mathcal{P} \cup \mathcal{Q}$ if
there exists $p=(x,w)\in  \mathcal{P}$ such that $y> w$ or
$q=(w,y)\in  \mathcal{Q}$ such that $x> w$.

Similarly, we say that $(x,y)\in L_k^2\setminus( \mathcal{P} \cup
\mathcal{Q} \cup [a,b]^2)$ is {\it below} $\mathcal{P} \cup
\mathcal{Q}$ if there exists a $p=(x,w)\in  \mathcal{P}$ such that
$y< w$ or a $q=(w,y)\in  \mathcal{Q}$ such that $x< w$.
\end{definition}

%The converse is also true with some restriction.
Using this terminology we can summarize the previous observations
and we get the following characterization. The next statement can be
seen as the analogue of theorem of Czoga\l a-Drewiak \cite[Theorem
3.]{Czogala1984} for finite chains.
\begin{theorem}\label{tfo1}
For every associative quasitrivial nondecreasing operation $F:L_k^2\to L_k$ there exist half-neutral elements $a,b\in L_k$ ($a\le b$) %such that $e$ and $f$ are the lower-half and upper-half-neutral element respectively.
and a downward-right path $\mathcal{P}=(p_j)_{j=1}^l$ (for some
$l\in\mathbb{N},l<k$) from $(1,k)$ to $(a,b)$. We denote the
reflection of $\mathcal{P}$ to the diagonal $\Delta_{L_k}$ by
$\mathcal{Q}=(q_j)_{j=1}^l$. Then for every
$(x,y)\not\in\mathcal{P}\cup\mathcal{Q}$ $$F(x,y)=\begin{cases}
x\vee y, &\textrm{ if } (x,y) \textrm{ is above } \mathcal{P} \cup \mathcal{Q}\\
x\wedge y, &\textrm{ if } (x,y) \textrm{ is below } \mathcal{P} \cup \mathcal{Q}\\
Proj_x \textrm{ or } Proj_y, &\textrm{ if } (x,y)\in [a,b]^2,
\end{cases}$$
and for every $(x,y)\in\mathcal{P}\cup\mathcal{Q}$
$$F(x,y)=\begin{cases}
%Proj_x \textrm{ or } Proj_y , &\textrm{ for every } (x,y)\in [a,b]^2,\\
x\wedge y  &\textrm{ if } (x,y)=p_i \textrm{ or } q_i \textrm{ and } \overline{p_i,p_{i+1}} \textrm{ is horizontal}, \\
x\vee y, &\textrm{ if } (x,y)=p_i \textrm{ or } q_i \textrm{ and } \overline{p_i,p_{i+1}} \textrm{ is vertical,}\\
x \textrm{ or } y, &\textrm{ if } (x,y)=p_i \textrm{ and } \overline{p_i,p_{i+1}} \textrm{ is diagonal,}\\
x \textrm{ or } y, &\textrm{ if } (x,y)=q_i \textrm{ and } \overline{q_i,q_{i+1}} \textrm{ is diagonal.}\\
\end{cases}$$

If $a$ is the lower half-neutral element $f$ and $b$ is the upper
half-neutral element $e$, then $F$ is $Proj_x$ on $[a,b]^2$,
otherwise it is $Proj_y$.

Moreover $F$ is symmetric expect on $[a,b]^2$ and at the points
$p_i\in \mathcal{P}$ and $q_i\in \mathcal{Q}$ where $\overline{p_i,
p_{i+1}}$ is diagonal ($i\in \{1, \dots,l-1\}$).

 \begin{figure}[ht!]
 \centering
\begin{tikzpicture}
%\node[] at (3,-0.3) {b};
\filldraw (0,3) circle[radius=1.5pt];

\filldraw (0.3,2.7) circle[radius=1.5pt]; \filldraw (0.6,2.7)
circle[radius=1.5pt]; \node[] at (0.6,2.4) {$\mathcal{P}$};

\filldraw (0.9,2.7) circle[radius=1.5pt]; \filldraw (1.2,2.4)
circle[radius=1.5pt];
%\filldraw[red] [->] (0.9,2.7) -- (0.9,2.4);
%\filldraw[red] [->] (0.9,2.7) -- (1.2,2.7);
%\filldraw[red] [->] (0.9,2.7) -- (1.15,2.45);

\filldraw(1.2,2.1) circle[radius=1.5pt];

%\node[] at (-0.3,3) {a};
%\filldraw (0,3) circle[radius=1.5pt];
%\node[] at (0.4,-0.3) {a};
%\filldraw (0.4,0) circle[radius=1.5pt];
\draw[] (0,3)--(3,3)--(3,0)--(0,0)--cycle; \draw[]
(0,3)--(0.3,2.7)--(0.6,2.7)--(0.9,2.7)--(1.2,2.4)--(1.2,2.1);
\draw[] (1.2,2.1)--(1.2,1.2)--(2.1,1.2)--(2.1,2.1)--cycle; \filldraw
(3,0) circle[radius=1.5pt]; \filldraw (2.7,0.3)
circle[radius=1.5pt]; \filldraw (2.7,0.6) circle[radius=1.5pt];
\node[] at (2.4,0.6) {$\mathcal{Q}$}; \filldraw (2.7,0.9)
circle[radius=1.5pt]; \filldraw (2.4,1.2) circle[radius=1.5pt];
\filldraw (2.1,1.2) circle[radius=1.5pt]; \draw[](3,0)--(2.7,0.3) --
(2.7,0.6)--
 (2.7,0.9)-- (2.4,1.2) -- (2.1,1.2);
\node[] at (0.8,0.8) {$x\wedge y$ }; \node[] at (2.55,2.55) {$x\vee
y$}; \node[] at (1.65,1.65) {$Proj$};

\end{tikzpicture}
\caption{Characterization of associative quasitrivial nondecreasing
operations on finite chains}
\end{figure}
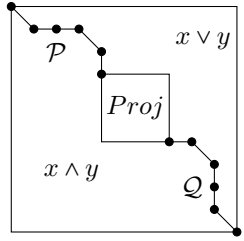
\end{theorem}

\begin{proof}
The statement is clearly follows from the Algorithm and the
definition of paths $\mathcal{P}$ and $\mathcal{Q}$.
\end{proof}

%\begin{corollary}
%Every associative quasitrivial nondecreasing binary operation $F:L_k^2\to L_k$ is symmetric except on $[a,b]^2$, where $a$ and $b$ the half-neutral elements of $F$ and on the paths $\mathcal{P}$ and $\mathcal{Q}$ associated to $F$ as in Theorem \ref{tfo1}.
%\end{corollary}

The converse statement can be formalized as follows. This statement
plays the role of theorem of Martin-Mayor-Torrens \cite[Theorem
4.]{Martin2003} for finite chains.

\begin{theorem}\label{thmchar}
Let $\mathcal{P}=(p_j)_{j=1}^l$ be a downward-right path in
$T_1\subset L_k^2$ from $(1,k)$ to $(a,b)$ $(a\le b)$ and let
$\mathcal{Q}=(q_j)_{j=1}^l$ be its reflection to the diagonal
$\Delta_{L_k}$. Let $F:L_k^2\to L_k$ be defined for every
$(x,y)\not\in\mathcal{P}\cup\mathcal{Q}$ as
$$F(x,y)=\begin{cases}
x\vee y, &\textrm{ if } (x,y) \textrm{ is above } \mathcal{P} \cup \mathcal{Q},\\
x\wedge y, &\textrm{ if } (x,y) \textrm{ is below } \mathcal{P} \cup \mathcal{Q},\\
Proj_x \textrm{ or } Proj_y \textrm{ (uniformly)}, &\textrm{ for every } (x,y)\in [a,b]^2. %\textrm{ if } (x,y)\in [a,b]^2.
\end{cases}$$
and for every $(x,y)\in\mathcal{P}\cup\mathcal{Q}$
$$F(x,y)=\begin{cases}
%Proj_x \textrm{ or } Proj_y , &\textrm{ for every } (x,y)\in [a,b]^2,\\
x\wedge y  &\textrm{ if } (x,y)=p_i \textrm{ or } q_i \textrm{ and } \overline{p_i,p_{i+1}} \textrm{ is horizontal}, \\
x\vee y, &\textrm{ if } (x,y)=p_i \textrm{ or } q_i \textrm{ and } \overline{p_i,p_{i+1}} \textrm{ is vertical,}\\
x \textrm{ or } y \textrm{ (arbitrarily) }, &\textrm{ if } (x,y)=p_i
\textrm{ and } \overline{p_i,p_{i+1}} \textrm{ is diagonal.}
\end{cases}$$
If $(x,y)=q_i$ and $\overline{q_i,q_{i+1}} \textrm{ (or equivalently
}\overline{p_i,p_{i+1}}) \textrm{ is diagonal,}$ then $F(x,y)\in
\{x, y\}$ and $F(x,y)\ne F(y,x)$ uniquely define $F(x,y)$. Then $F$
is associative quasitrivial and nondecreasing.
%Let also given the value of $F$ in the top elements of a diagonal step of $\mathcal{P}$ and in $(a,b)$ if $\mathcal{P}$ is ended in $(a,b)$ and $a\ne b$. Then a associative quasitrivial nondecreasing operation $F$ is uniquely determined.
%Moreover, if the path $\mathcal{P}$ contains moves diagonally, then in the top-left element $(x,y)\in \mathcal{P}$ of that move the value of $F$ can be chosen arbitrarily $x_{p_i}$ or $y_{p_i}$.
\end{theorem}
\proof It is clear that $F$ is defined for every $(x,y)\in L_k^2$
and $F$ is quasitrivial and nondecreasing. Now we show that $F$ is
associative. If it is not the case, then by Lemma \ref{lpic}, one of
the cases of Figure \ref{fig111} is realized. Let $u,v,w \in L_k$
($u<v<w$) denote the elements where its realized. Clearly $F(u,w)\ne
F(w,u)$ and $F$ is not a projection on $[u,w]^2$. Thus, by the
definition of $F$, it follows that $(u,w)\in \mathcal{P}$ and
$(w,u)\in \mathcal{Q}$. Hence $p_i=(u,w)$ for some $i=\{1, \dots,
l-1\}$ and $\overline{p_i, p_{i+1}}$ is diagonal. Thus we have one
of the following situation (Figure \ref{figrem}).
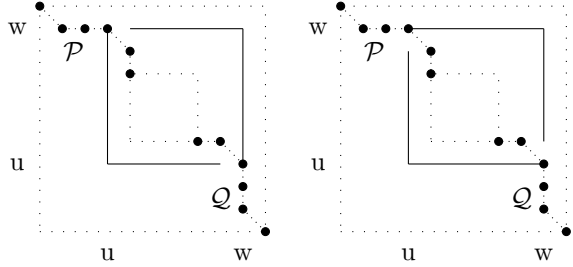
\begin{figure}[ht]
\begin{tikzpicture}
 %C/.style = {circle,thick,draw, inner sep=2pt},
%\draw (0,0)-- (0,3)--(3,3)--(3,0) -- cycle;
%\filldraw[red]
%\draw (0,0)--(3,3);
%\node[] at (2,7.3) {};

\draw (0.9,2.7)-- (0.9,0.9); \draw (1.2,2.7)-- (2.7,2.7); \draw
(2.7,0.9)--(2.7,2.7); \draw (2.4,0.9)--(0.9, 0.9); \node[] at
(0.9,-0.3) {u};
%\node[] at (1.2,3.3) {u+1};
\node[] at (2.7,-0.3) {w};
%\node[] at (2.4,3.3) {w-1};

\node[] at (-0.3,2.7) {w}; \node[] at (-0.3,0.9) {u};
%\draw[loosely dotted]
%(0,3)--(3,3)--(3,0)--(0,0)--cycle;
\filldraw (0,3) circle[radius=1.5pt];

\filldraw (0.3,2.7) circle[radius=1.5pt]; \filldraw (0.6,2.7)
circle[radius=1.5pt]; \node[] at (0.45,2.4) {$\mathcal{P}$};

\filldraw (0.9,2.7) circle[radius=1.5pt]; \filldraw (1.2,2.4)
circle[radius=1.5pt];
%\filldraw[red] [->] (0.9,2.7) -- (0.9,2.4);
%\filldraw[red] [->] (0.9,2.7) -- (1.2,2.7);
%\filldraw[red] [->] (0.9,2.7) -- (1.15,2.45);

\filldraw(1.2,2.1) circle[radius=1.5pt];

%\node[] at (-0.3,3) {a};
%\filldraw (0,3) circle[radius=1.5pt];
%\node[] at (0.4,-0.3) {a};
%\filldraw (0.4,0) circle[radius=1.5pt];
\draw[loosely dotted] (0,3)--(3,3)--(3,0)--(0,0)--cycle;
\draw[dotted]
(0,3)--(0.3,2.7)--(0.6,2.7)--(0.9,2.7)--(1.2,2.4)--(1.2,2.1);
\draw[loosely dotted]
(1.2,2.1)--(1.2,1.2)--(2.1,1.2)--(2.1,2.1)--cycle; \filldraw (3,0)
circle[radius=1.5pt]; \filldraw (2.7,0.3) circle[radius=1.5pt];
\filldraw (2.7,0.6) circle[radius=1.5pt]; \node[] at (2.4,0.45)
{$\mathcal{Q}$}; \filldraw (2.7,0.9) circle[radius=1.5pt]; \filldraw
(2.4,1.2) circle[radius=1.5pt]; \filldraw (2.1,1.2)
circle[radius=1.5pt]; \draw[dotted](3,0)--(2.7,0.3) -- (2.7,0.6)--
 (2.7,0.9)-- (2.4,1.2) -- (2.1,1.2);

\draw (4.9,2.4)-- (4.9,0.9); \draw (4.9,2.7)-- (6.7,2.7); \draw
(6.7,1.2)--(6.7,2.7); \draw (6.7,0.9)--(4.9, 0.9); \node[] at
(4.9,-0.3) {u}; \node[] at (6.7,-0.3) {w}; \node[] at (3.7,2.7) {w};
\node[] at (3.7,0.9) {u};
%\draw[loosely dotted]
%(0,3)--(3,3)--(3,0)--(0,0)--cycle;
\filldraw (4,3) circle[radius=1.5pt];

\filldraw (4.3,2.7) circle[radius=1.5pt]; \filldraw (4.6,2.7)
circle[radius=1.5pt]; \node[] at (4.45,2.4) {$\mathcal{P}$};

\filldraw (4.9,2.7) circle[radius=1.5pt]; \filldraw (5.2,2.4)
circle[radius=1.5pt];
%\filldraw[red] [->] (0.9,2.7) -- (0.9,2.4);
%\filldraw[red] [->] (0.9,2.7) -- (1.2,2.7);
%\filldraw[red] [->] (0.9,2.7) -- (1.15,2.45);

\filldraw(5.2,2.1) circle[radius=1.5pt];

%\node[] at (-0.3,3) {a};
%\filldraw (0,3) circle[radius=1.5pt];
%\node[] at (0.4,-0.3) {a};
%\filldraw (0.4,0) circle[radius=1.5pt];
\draw[loosely dotted] (4,3)--(7,3)--(7,0)--(4,0)--cycle;
\draw[dotted]
(4,3)--(4.3,2.7)--(4.6,2.7)--(4.9,2.7)--(5.2,2.4)--(5.2,2.1);
\draw[loosely dotted]
(5.2,2.1)--(5.2,1.2)--(6.1,1.2)--(6.1,2.1)--cycle; \filldraw (7,0)
circle[radius=1.5pt]; \filldraw (6.7,0.3) circle[radius=1.5pt];
\filldraw (6.7,0.6) circle[radius=1.5pt]; \node[] at (6.4,0.45)
{$\mathcal{Q}$}; \filldraw (6.7,0.9) circle[radius=1.5pt]; \filldraw
(6.4,1.2) circle[radius=1.5pt]; \filldraw (6.1,1.2)
circle[radius=1.5pt]; \draw[dotted](7,0)--(6.7,0.3) -- (6.7,0.6)--
 (6.7,0.9)-- (6.4,1.2) -- (6.1,1.2);

 \end{tikzpicture}
\caption{Two remaining cases}\label{figrem}
\end{figure}

Therefore, since $u<v<w$, it follows that $F(u,v)\ne v, F(v,u)\ne v,
F(w,v)\ne v, F(v,w)\ne v$. Hence, none of the cases of Figure
\ref{fig111} can be realized. Thus $F$ is associative. \qed
\begin{remark}
According to Theorems \ref{tfo1} and \ref{thmchar} it is clear that
there is a surjection from the set of associative quasitrivial
nondecreasing operations defined on $L_k$ to the downward-right
paths defined on $T_1$ and started at $(1,k)$ (and ended somewhere
in $T_1$). This surjection is a bijection if and only if the path
$\mathcal{P}$ does not contain a diagonal move and $a=b$. This
condition is equivalent that $F$ is symmetric (and has a neutral
element).
\end{remark}
\begin{corollary}\label{cchar}
Let $F:L_k^2\to L_k$ be an associative quasitrivial nondecreasing
operation. If $F$ is symmetric, then it is uniquely determined by a
downward-right path $\mathcal{P}$ containing only horizontal and
vertical line segments and it starts at $(1,k)$ and reaches the
diagonal $\Delta_{L_k}$.
\end{corollary}
As a consequence of the previous corollary we obtain the result of
\cite[Theorem 4.]{Beats2009} (see also \cite[Theorem
14.]{Jimmy2017}).
\begin{corollary}\label{ccard}
The number of associative quasitrivial nondecreasing symmetric
operation defined on $L_k$ is $2^{k-1}$.
\end{corollary}
\begin{proof}
Every path from $(1,k)$ to the diagonal  $\Delta_{L_k}$ using right
or downward moves contains $k$ points. According to Corollary
\ref{cchar}, in each point of the path, except the last one, we have
two options which direction we move further. This immediately
implies that the number of associative quasitrivial nondecreasing
symmetric operation defined on $L_k$ is $2^{k-1}$.\end{proof}

In Theorem \ref{tnumb}, as an application of the results of this section,  we calculate the number of associative quasitrivial nondecreasing operations defined on $L_k$ and also the number of associative quasitrivial nondecreasing operations on $L_k$ that have neutral elements. %This is a kind of an analogue of Corollary \ref{ccard}.

\begin{remark}
\begin{enumerate}
\item[(a)]
We note that from the proof of Lemma \ref{limp}  throughout this section  %Theorem \ref{thmchar}
we essentially use that $F$ is defined on a finite chain.
\item[(b)] In the continuous case \cite{Czogala1984, Martin2003} and also in the symmetric case \cite{Beats2009, Jimmy2017} it is always possible to define a one variable function $g$, such that the extended graph of $g$ separates the points of the domain of the binary operation $F$ into two parts where $F$ is a minimum and a maximum, respectively. Now the paths $\mathcal{P}$ and $\mathcal{Q}$ play the role of the extended graph of $g$. Because of the diagonal moves of the path $\mathcal{P}$, it does not seems so clear how such a 'separating' function can be defined in the non-symmetric discrete case.
\end{enumerate}
\end{remark}
%Since $\mathcal{Q}$ is a reflection of $\mathcal{P}$ to the diagonal $\Delta_{L_k}$ and $\mathcal{P}$ is a downward-right path,
%it is possible to define a nondecreasing operation $g:([1,e]\cup [f,k])^2\to ([1,e]\cup [f,k])$ such that extended graph $\Gamma_g$ of $g$ is symmetric to the diagonal $\Delta_{L_k}$.
%Indeed, let $g$ be defined for every $m\in ([1,e]\cup [f,k]$ as $g(m)=\max{(p_i)_2: (p_i)_2=m}$
% ez nem teljesen jo igy.
%\begin{proposition}
%    For every associative quasitrivial nondecreasing operation $F:L_k^2\to L_k$ we can associative a
%\end{proposition}
%Thus each step the path $\mathcal{P}$
%Dinamical way.

%\begin{theorem}\label{thmchar}
%Characterization, with paths arrows
%\end{theorem}
%We can think on this path as a extended graph of a nonincreasing operation. Then we get
%\begin{corollary}
% DeBeats.
%\end{corollary}

%++++++++++++++++++++++++++++

%Result connected to DeBEATS et al.

%++++++++++++++++++++++++++++

%Number of associative quasitrivial nondecreasing operations and the number of associative quasitrivial nondecreasing operations having a neutral element.

%\begin{remark}

%\end{remark}

\section{Bisymmetric operations}\label{s5}
In this section we show a characterization of bisymmetric quasitrivial nondecreasing binary operations based on the previous section.% We start our investigation with the binary case.
%\subsection{Binary case}
The following statement was proved as \cite[Lemma 22.]{Jimmy2017}.
\begin{lemma}\label{lba}
Let $X$ be an arbitrary set and $F:X^2 \to X$ be an operation. Then
the following assertions hold.
\begin{enumerate}
    \item[(a)] If $F$ is bisymmetric and has a neutral element, then it is associative and symmetric.
    \item[(b)] If $F$ is bisymmetric and quasitrivial, then $F$ is associative.
    \item[(c)] If $F$ is associative and symmetric, then it is bisymmetric.
\end{enumerate}
\end{lemma}

%We note that this is not the case in higher dimension (see Example \ref{exbi}). % every weighted arithmetic sum of the form   is bisymmetric
%but only the sum is associative.
Using also the results of Section \ref{s4} we get the following
statement.
\begin{theorem}\label{thmbi}
Let $F:L_k^2\to L_k$ be a bisymmetric quasitrivial nondecreasing
operation. Then there exists the upper half-neutral element $e$ and
the lower half-neutral element $f$ and $F$ is symmetric on
$(L_k\setminus[e\wedge f,e\vee f])^2$.
\end{theorem}
\begin{proof}
According to Lemma \ref{lba}(b), every quasitrivial bisymmetric
operations are associative. Thus, by Proposition \ref{cor} it has an
upper and lower half-neutral element ($e$ and $f$, respectively).

Let us assume that $e\le f$ (the case when $f\le e$ can be handled
similarly).

If there exists $u,v\in L_k$ such that $u<v$, $F(u,v)\ne F(v,u)$, then %We also assume that $e\le f$.
by Corollary \ref{corcases}, either $u,v \in [e,f]$ (then we do not
need to prove anything) or  $u,v\not\in [e,f]$. Moreover, if
$u,v\not\in [e,f]$, then Lemma \ref{limp}(a) or Remark \ref{rimp}(a)
holds. The existence of $e$ implies that $v-u\ge 2$.

If \begin{equation*}\label{equv} u=F(u,v) \ne F(v,u)=v
\end{equation*} is satisfied, then Lemma \ref{limp} (a) holds (i.e,
$F(x,v)=v$ if $x\in [u+1,v]$ and $F(y,u)=u$ if $y\in[u,v-1]$). Since
$v-u\ge 2$, $u+1\le v-1$, hence $F(u+1,u)=u$. On the other hand, $F$
is monotone and idempotent, thus by Observation \ref{l1}, $F(v,t)=v$
and $F(u,t)=u$ for all $t\in [u,v]$. Using bisymmetric equation we
get the following
\begin{equation*}
u=F(u,v)=F(F(u+1,u),F(v,v-1))=F(F(u+1,v), F(u,v-1))=F(v,u)=v,
\end{equation*}
which is a contradiction.

Similarly, if \begin{equation*}\label{eqvu} v=F(u,v) \ne F(v,u)=u
\end{equation*} is satisfied, then Remark \ref{rimp} (a) holds (i.e,
$F(v,x)=v$ if $x\in [u+1,v]$ and $F(u,y)=u$ if $y\in[u,v-1]$). Since
$v-u\ge 2$, $u+1\le v-1$, hence $F(v-1,v)=v$. Applying Observation
\ref{l1} again, we have $F(t,v)=v$ and $F(t,u)=u$ for all $t\in
[u,v]$. Using bisymmetric equation we get a contradiction as
\begin{equation*}
u=F(v,u)=F(F(v-1,v),F(u,u+1))=F(F(v-1,u), F(v,u+1))=F(u,v)=v.
\end{equation*}
%Similar argument shows the case when $e\le f$ and $b= F(a,b)\ne F(a,b)=a$.
\end{proof}
% Characterization.

Applying Theorem \ref{thmbi} we get the following characterization.

\begin{theorem}\label{tbfo1}
Let $F:L_k^2\to L_k$ be a quasitrivial nondecreasing operation. Then
$F$ is bisymmetric if and only if there exists $a,b\in L_k$ ($a\le
b$) and a downward-right path $\mathcal{P}=(p_j)_{j=1}^l$ (for some
$l\in \mathbb{N}$) from $(1,k)$ to $(a,b)$ containing only
horizontal and vertical line segments such that for every
$(x,y)\not\in\mathcal{P}\cup\mathcal{Q}$
\begin{equation}\label{eqbi1}F(x,y)=\begin{cases}
x\vee y, &\textrm{ if } (x,y) \textrm{ is above } \mathcal{P} \cup \mathcal{Q},\\
x\wedge y, &\textrm{ if } (x,y) \textrm{ is below } \mathcal{P} \cup \mathcal{Q},\\
Proj_x \textrm{ or } Proj_y \textrm{ (uniformly)}, &\textrm{ for every } (x,y)\in [a,b]^2. %\textrm{ if } (x,y)\in [a,b]^2.
\end{cases}
\end{equation}
and for every $(x,y)\in\mathcal{P}\cup\mathcal{Q}$
\begin{equation}\label{eqbi2}F(x,y)=\begin{cases}
%Proj_x \textrm{ or } Proj_y , &\textrm{ for every } (x,y)\in [a,b]^2,\\
x\wedge y  &\textrm{ if } (x,y)=p_i \textrm{ or } q_i \textrm{ and } \overline{p_i,p_{i+1}} \textrm{ is horizontal}, \\
x\vee y, &\textrm{ if } (x,y)=p_i \textrm{ or } q_i \textrm{ and }
\overline{p_i,p_{i+1}} \textrm{ is vertical,}
\end{cases}
\end{equation}
where $\mathcal{Q}=(q_j)_{j=1}^l$ is the reflection of $\mathcal{P}$
to the diagonal $\Delta_{L_k}$.

In particular, $F$ is symmetric on $L_k^2\setminus[a,b]^2$ and one
of the projections on $[a,b]^2$.

 \begin{figure}[ht!]
 \centering
\begin{tikzpicture}
%\node[] at (3,-0.3) {b};
\filldraw (0,3) circle[radius=1.5pt]; \filldraw (0.3,3)
circle[radius=1.5pt]; \filldraw (0.3,2.7) circle[radius=1.5pt];
\filldraw (0.6,2.7) circle[radius=1.5pt]; \node[] at (0.6,2.4)
{$\mathcal{P}$}; \filldraw (0.9,2.4) circle[radius=1.5pt]; \filldraw
(0.9,2.7) circle[radius=1.5pt]; \filldraw (1.2,2.4)
circle[radius=1.5pt];
%\filldraw[red] [->] (0.9,2.7) -- (0.9,2.4);
%\filldraw[red] [->] (0.9,2.7) -- (1.2,2.7);
%\filldraw[red] [->] (0.9,2.7) -- (1.15,2.45);

\filldraw(1.2,2.1) circle[radius=1.5pt];

%\node[] at (-0.3,3) {a};
%\filldraw (0,3) circle[radius=1.5pt];
%\node[] at (0.4,-0.3) {a};
%\filldraw (0.4,0) circle[radius=1.5pt];
\draw[] (0,3)--(3,3)--(3,0)--(0,0)--cycle; \draw[]
(0,3)--(0.3,3)--(0.3,2.7)--(0.6,2.7)--(0.9,2.7)--(0.9,2.4)--(1.2,2.4)--(1.2,2.1);
\draw[] (1.2,2.1)--(1.2,1.2)--(2.1,1.2)--(2.1,2.1)--cycle; \filldraw
(3,0) circle[radius=1.5pt]; \filldraw (3,0.3) circle[radius=1.5pt];

\filldraw (2.7,0.3) circle[radius=1.5pt]; \filldraw (2.7,0.6)
circle[radius=1.5pt]; \node[] at (2.4,0.6) {$\mathcal{Q}$};
\filldraw (2.7,0.9) circle[radius=1.5pt]; \filldraw (2.4,0.9)
circle[radius=1.5pt]; \filldraw (2.4,1.2) circle[radius=1.5pt];
\filldraw (2.1,1.2) circle[radius=1.5pt];
\draw[](3,0)--(3,0.3)--(2.7,0.3) -- (2.7,0.6)--
 (2.7,0.9)--(2.4,0.9) --(2.4,1.2) -- (2.1,1.2);
\node[] at (0.8,0.8) {$x\wedge y$ }; \node[] at (2.55,2.55) {$x\vee
y$}; \node[] at (1.65,1.65) {$Proj$};

\end{tikzpicture}
\caption{Characterization of bisymmetric quasitrivial nondecreasing
operations on finite chains}
\end{figure}
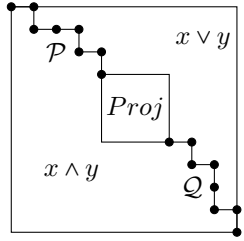
\end{theorem}

\begin{proof}
(Necessity) Since $F$ is bisymmetric and quasitrivial, by Lemma \ref{lba}(b), $F$ is associative. By Theorem \ref{tfo1}, there exist half-neutral elements $a,b\in L_k$ ($a<b$) %such that $e$ and $f$ are the lower-half and upper-half-neutral element respectively.
and a downward-right path $\mathcal{P}$ from $(1,k)$ to $(a,b)$. By
Theorem \ref{tbfo1}, $F$ is symmetric on $L_k^2\setminus[a,b]^2$.
Thus $\mathcal{P}$ does not contain a diagonal line segment. Hence,
applying again Theorem \ref{tfo1} we get that $F$ satisfies
\eqref{eqbi1} and \eqref{eqbi2}.

(Sufficiency) The operation $F$ defined by \eqref{eqbi1} and
\eqref{eqbi2} satisfies the conditions of Theorem \ref{thmchar},
thus $F$ is quasitrivial nondecreasing and associative. Now we show
that $F$ is bisymmetric (i.e, $\forall u,v,w,z\in L_k$
\begin{equation}\label{eqbia}
F(F(u,v),F(w,z))=F(F(u,w),F(v,z)).)\end{equation} Let us assume that
$F(x,y)=Proj_x$ on $[a,b]^2$ (for $F(x,y)=Proj_y$ on $[a,b]^2$ the
proof is analogue). By Corollary \ref{corcases}, this implies that
$a=f$ and $b=e$ ($f<e$) and, by Proposition \ref{cor}, it is clear
that \begin{equation}\label{eqproj}F(x,y)=x \ \forall x \in L_k,
\forall y \in [a,b].\end{equation} Since $F$ is associative, we have
$$F(F(u,v),F(w,z))=F(F(F(u,v),w),z)=F(F(u, F(v,w)),z)$$
and
$$F(F(u,w),F(v,z))=F(F(F(u,w),v),z)=F(F(u, F(w,v)),z).$$

If $F(v,w)=F(w,v)$, then \eqref{eqbia} follows and we are done.

If  $F(v,w)\ne F(w,v)$, then $v,w\in [a,b]^2$ and, since
$F(x,y)=Proj_x$ on $[a,b]^2$, $F(v,w)=v$ and $F(w,v)=w$. Then, by
\eqref{eqproj},
\begin{align*}
&F(F(u, F(v,w)),z)=F(F(u,v),z)=F(u,z),\\
&F(F(u, F(v,w)),z)=F(F(u,w),z)=F(u,z).
\end{align*}
Thus $F$ is bisymmetric.
\end{proof}
\begin{remark}
\begin{enumerate}
    \item[(a)]  There is a one-to-one correspondence between downward-right paths containing only vertical and horizontal line segments and the quasitrivial nondecreasing bisymmetric operations if we fix that the operation is $Proj_x$ on $[a,b]^2$ ($a$ and $b$ are the half neutral-elements of the operation).
    The same is true, if the operation is $Proj_y$ on $[a,b]^2$.
\item[(b)]
The nondecreasing assumption can be substituted by monotonicity. Indeed, %by Lemma \ref{lba}(b), every bisymmetric quasitrivial operation is associative. By
by Corollary \ref{cnem}, monotonicity is equivalent with
nondecreasingness for quasitrivial operations.
\end{enumerate}
\end{remark}

\section{The number of operations of given class}\label{s6}

This section is devoted to calculate the number of associative
quasitrivial nondecreasing operations. Byproduct of the following
argument we also consider the number of associative quasitrivial
nondecreasing operations having neutral elements. With the same
technique one can easily deduce the number of bisymmetric
quasitrivial nondecreasing binary operations (see Proposition
\ref{binumb}).

\begin{theorem}\label{tnumb}
Let $A_k$ denote the number of associative quasitrivial
nondecreasing operations defined on $L_k$ and $B_k$ denote the
number of associative quasitrivial nondecreasing operations defined
on $L_k$ and having neutral elements. Then
$$A_k=\frac{1}{6}\big((2+\sqrt{3})(1+\sqrt{3})^k+(2-\sqrt{3})(1-\sqrt{3})^k-4\big),
$$
$$B_k=\frac{1}{2\cdot\sqrt{3}}\big((1+\sqrt{3})^k-(1-\sqrt{3})^k\big).$$

\end{theorem}

The following observations show that these numbers are related to
the downward-right path $\mathcal{P}=(p_j)_{j=1}^l$ (for some $l\le
k$) in $T_1$ starting from $(1,k)$. Let $m_{\mathcal{P}}$ be the
number of diagonal line segments $\overline{p_i,p_{i+1}} \in
\mathcal{P}$ ($i\in\{1, \dots, l-1\}$). We say that the
downward-right path $\mathcal{P}$ is {\it weighted} with weight
$2^{m_{\mathcal{P}}}$.

\begin{lemma}\label{lcount1}
\begin{enumerate}
    \item[(a)] $B_k$ is the sum of the weights of weighted paths that starts at $(1,k)$ and reaches $\Delta_{L_k}$.

    \item[(b)] $A_k+B_k$ is twice the sum of the weights of weighted paths in $T_1$ that starts at $(1,k)$ and ends at any point of $T_1$.
 \end{enumerate}
\end{lemma}

\begin{proof}\begin{enumerate}
\item[(a)]
Applying Theorem \ref{tfo1}, it is clear that if an associative
quasitrivial nondecreasing binary operation $F$ has a neutral
element, then the downward-right path $\mathcal{P}$ defined for $F$
reaches the diagonal $\Delta_{L_k}$. By Theorem \ref{thmchar}, there
can be defined $2^{m_{\mathcal{P}}}$ different operations for a
given path $\mathcal{P}$ that reaches the diagonal, since we have a
choice in each case when the path contains a diagonal line segment.
This show the first part of the statement.
\item[(b)]
This statement follows from the fact that for any associative
quasitrivial nondecreasing operation $F$ one can define a
downward-right path which starts at $(1,k)$ and ends somewhere in
$T_1$. If its end in $(a,b)$ where $a< b$ (not on $\Delta_{L_k}$),
then $F$ is one of the projections in $[a,b]^2$, and $a$ and $b$ are
the half-neutral elements of $F$. This makes the extra 2 factor in
the statement.

Let $\Pi_1$ denote set of the weighted paths in $T_1$ that starts at
$(1,k)$ and ends at $(a,b) $ where $a< b$. Similary, $\Pi_2$ denote
the set of weighted paths that starts at $(1,k)$ and reaches
$\Delta_{L_k}$. Hence,
\begin{equation*}
%\begin{split}
    A_k
    =2\cdot\sum_{\mathcal{P}\in \Pi_1}2^{m_{\mathcal{P}}}+\sum_{\mathcal{P}\in \Pi_2}2^{m_{\mathcal{P}}} %&\#\{\textrm{weighted downward-right paths from } (1,k) \textrm{ to } (a,b), \textrm{ where }a,b\in L_k, a<b\}\\
   % +&\#\{\textrm{weighted downward-right paths from } (1,k) \textrm{ to the diagonal } \Delta_{L_k}\},
%\end{split}
\end{equation*}
%where $\#$ denote the number of the corresponding set.
  According to the (a) part
\begin{equation*}
    B_k=\sum_{\mathcal{P}\in \Pi_2}2^{m_{\mathcal{P}}}.%\#\{\textrm{weighted downward-right paths from } (1,k) \textrm{ to the diagonal } \Delta_{L_k}\}
\end{equation*} Adding these equations, we get the statement for $A_k+B_k$.
\end{enumerate}
\end{proof}
Now we present a recursive formula for $A_k$ and $B_k$.
\begin{lemma}\label{lnumb}
\begin{enumerate}
\item[(a)]
$B_1=1$, $B_2=2$ and $B_k=2\cdot B_{k-1}+2\cdot B_{k-2}$ for every
$k\ge 3$.
\item[(b)]
$A_k=2\sum_{i=1}^k B_{i}-B_k$ for every $k\in \mathbb{N}$.
\end{enumerate}
\end{lemma}
\begin{proof}
\begin{enumerate}
    \item[(a)]  $B_1=1$, $B_2=2$ are clear.
    The recursive formula follows from the Algorithm presented in Section \ref{s4} and the definition of downward-right path $\mathcal{P}=(p_j)_{j=1}^l$.  Now we assume that $k\ge 3$. If $\overline{p_1, p_2}$ is horizontal or vertical, then Case I. (a) or (b) of the Algorithm holds (see also Figure \ref{figmove1}). Thus we reduce the square $Q_1$ of size $k$ to a square $Q_2$ of size $k-1$.  If $\overline{p_1, p_2}$ is diagonal, then Case III (a) or (b) holds (see also Figure \ref{figmove3}). Thus we reduce the square $Q_1$ of size $k$ to a square $Q_2$ of size $k-2$. By definition, the number of associative quasitrivial nondecreasing operations having neutral elements defined on a square of size $k$ is $B_k$. Thus we get that $B_k=2\cdot B_{k-1}+2\cdot B_{k-2}$.
    \item[(b)] This follows from Lemma \ref{lcount1} (b) and the fact that
    'sum of the weights of weighted paths from $(1, k)$ to any point of $T_1$' is exactly $\sum_{i=1}^k B_i$. Indeed, let $s\in\{1, \dots, k\}$ be fixed. Then $B_s$ is equal to the sum of the weights of weighted paths $\mathcal{P}$ that starts at $(1,k)$ and ends at $(a,b)$  where $b-a=s$.
\end{enumerate}
\end{proof}

%By Lemma \ref{lnumb}, we can calculate the value of $B_k$ and $A_k$.

{\it Proof of Theorem \ref{tnumb}.} We use a standard method of
second-order linear recurrence equations for the formula of Lemma
\ref{lnumb} (a). Therefore,
$$B_k=c_1\cdot(\alpha_1)^k+c_2(\alpha_2)^k,$$ where $\alpha_1,
\alpha_2$ ($\alpha_1< \alpha_2$) are the solutions of the equation
$x^2-2x-2=0$. Thus, $\alpha_1=1-\sqrt{3}, \alpha_2=1+\sqrt{3}$. By
the initial condition $B_1=1$ and $B_2=2$, we get that
$c_1=-c_2=\frac{1}{2 \sqrt{3}}$. Thus,
$$B_k=\frac{1}{2\cdot\sqrt{3}}\big((1+\sqrt{3})^k-(1-\sqrt{3})^k\big).$$

According to Lemma \ref{lnumb} (b), $A_k$ can be calculated as
$2\cdot\sum_{i=1}^k B_i-B_k$.

This provides that
\begin{equation*}
A_k=%\frac{1}{3}\big((1+\sqrt{3})^{k+1}+(1-\sqrt{3})^{k+1}-2\big)-\frac{1}{2\cdot\sqrt{3}}\big((1+\sqrt{3})^k-(1-\sqrt{3})^k\big)\\
%&=\frac{1}{6}\big(2\cdot(1+\sqrt{3})^{k+1}+2\cdot(1-\sqrt{3})^{k+1}-\sqrt{3}\cdot(1+\sqrt{3})^k+\sqrt{3}\cdot(1-\sqrt{3})^k-4\big)\\
\frac{1}{6}\big((2+\sqrt{3})(1+\sqrt{3})^k+(2-\sqrt{3})(1-\sqrt{3})^k-4\big)
\end{equation*}
\qed

Here we present a list of the first 10 value of $A_k$: $A_1=1$,
$A_2=4$, $A_3=12, A_4=34, A_5=94, A_6=258, A_7=706, A_8=1930,
A_9=5274, A_{10}=14410$.

By Theorem \ref{t2}, we get the similar results for the $n$-ary
case.
\begin{corollary}
\begin{enumerate}
    \item [(a)]
The number of associative quasitrivial nondecreasing operations
$F:L_k^n\to L_k$ ($k\in \mathbb{N}$) having neutral elements is
$$\frac{1}{2\cdot\sqrt{3}}\big((1+\sqrt{3})^k-(1-\sqrt{3})^k\big),$$
\item[(b)]
The number of associative quasitrivial nondecreasing operations
$F:L_k^n\to L_k$ ($k\in \mathbb{N}$)  is
$$\frac{1}{6}\big((2+\sqrt{3})(1+\sqrt{3})^k+(2-\sqrt{3})(1-\sqrt{3})^k-4\big).
$$
\end{enumerate}
\end{corollary}

\begin{proposition}\label{binumb}
Let $C_k$ denote number of bisymmetric quasitrivial nondecreasing
binary operations defined in $L_k$ and $D_k$ denote the number of
bisymmetric quasitrivial nondecreasing binary operations having
neutral elements. Then
$$D_k=2^{k-1},$$
$$C_k=3\cdot 2^{k-1}-2.$$
\end{proposition}
\begin{proof}
\begin{enumerate}
\item[(a)]
By Lemma \ref{lba} and Proposition \ref{prop:ane}, bisymmetric quasitrivial nondecreasing binary operations having neutral elements defined on $L_k$ are exactly the associative quasitrivial symmetric nondecreasing binary operations. Thus by Corollary \ref{ccard}, we get that $D_k=2^{k-1}$.%a symmetric and associative operation is bisymmetric. According to Proposition \ref{prop:ane}, a quasitrivial symmetric nondecreasing operation defined on a finite chain has a neutral element. Thus, .
%Hence, the statement follows from Corollary \ref{ccard}.
\item[(b)]
Same argument as in Lemma \ref{lnumb}(b) shows that
$C_k=2\sum_{i=1}^k{D_i}-D_k$. Using this we get that $C_k=2\cdot
(2^{k}-1)-2^{k-1}=3\cdot 2^{k-1}-2.$
\end{enumerate}
\end{proof}

\begin{remark}
During the finalization of this paper the author have been informed
that  Miguel Couceiro, Jimmy Devillet and Jean-Luc Marichal found an
alternative and independent approach for similar estimations in
their upcoming paper \cite{Jimmy}.
\end{remark}

\section{Open problems and further perspectives}\label{s7}
First we summarize the most important results of our paper. In this
article we introduced a geometric interpretation of quasitrivial
nondecreasing associative binary operations. We gave a
characterization of such operations on finite chains using
downward-right paths. Combining this with a reducibility argument we
provided characterization for the $n$-ary analogue of the problem.
As a remarkable application of our visualization method we gave
characterization of bisymmetric quasitrivial nondecreasing binary
operation on finite chains. As a byproduct of our argument we
estimated the number of operations belonging to these classes.

These results initiate the following open problems.
\begin{enumerate}
    \item Characterize the $n$-ary bisymmetric quasitrivial nondecreasing operations. If these operations are also associative, then we can apply reducibility to deduce a characterization for them. On the other hand if $n\ge 3$, then not all of such operations are associative as the following example shows. Let $F\colon X^n\to X$ ($n\ge 3$) be the projection on the $i^{th}$ coordinate where $i$ is neither 1 or $n$. Then it is easy to show that it is bisymmetric quasitrivial nondecreasing but not associative. %The original construction of the author was much complicated, this example was found independently by Jimmy Devillet and the anonymous referee.
    \item Find a visual characterization of associative idempotent nondecreasing operations. Quasitrivial operations are automatically idempotent. Since idempotent operations are essentially important in fuzzy logic, this problem has its own interest.
\end{enumerate}

%It is an open question how can we characterize bisymmetric quasitrivial $n$-ary operations for $n\ge 3$ either with or without the assumption of nondecreasingness.

\section*{Acknowledgements} The author would like to thank Jimmy Devillet and the anonymous referee for the example given in the first open problem in Section 7.
This research is supported by the internal research project
R-AGR-0500 of the University of Luxembourg. The author was partially
supported by the Hungarian Scientific Research Fund (OTKA) K104178.

\section*{Appendix}
This section is devoted to prove the analogue of Corollary
\ref{cnem}. As it was already mentioned in Remark \ref{rnem}, the
proof is just a slight modification of the proof of \cite[Theorem
3.2]{KS}. The difference is based on the following easy lemma.
\begin{lemma}\label{lASP}
Let $X$ be a chain and $F: X^n\to X$ be an associative monotone
operation. Then $F$ is non-decreasing in the first and the last
variable.
\end{lemma}

\begin{proof}
The argument for the first and for the last variable is similar. We
just consider it for the first variable. From the definition of
associativity it is clear that an associative operation $F\colon X^n
\to X$ is satisfies
\begin{equation}\label{ASP}
    \begin{split}
    &F(F(x_1,\dots, x_n),x_{n+1}, \dots,  x_{2n-1})= \\ & F(x_1, F(x_{2}, \dots, x_{n+1}), x_{n+2}, \dots, x_{2n-1}).
    \end{split}
    \end{equation}
    for every $x_1, \dots, x_{2n-1}\in X$.
Now let us fix $x_2, \dots, x_{2n-1}\in X$  and define $$h(x)=F(F(x,
x_2,\dots ,x_n),x_{n+1}, \dots, x_{2n-1}).$$ The operation $F$ is
monotonic in the first variable thus it is clear that $h(x)$ is
nondecreasing, since we apply $F$ twice when $x$ is in the first
variable. Then using \eqref{ASP} we get that $F$ must be
nondecreasing in the first variable.
\end{proof}

As it was also mentioned in \cite{KS} the following condition is an
easy application of \cite[Theorem 1.4]{A} using the statement
therein for $A_2=\emptyset$.

\begin{theorem}\label{thmAkk}
Let $X$ be an arbitrary set. Suppose $F:X^n\to X$ be a quasitrivial
associative operation. If $F$ is not derived from a binary operation
$G$, then $n$ is odd and there exist $b_1 , b_2$ $(b_1\ne b_2)$ such
that for any $a_1, \dots, a_n \in \{b_1, b_2\}$
\begin{equation}\label{eqbk}
    F(a_1, \dots, a_n)=b_i~~(i=\{1,2\}),
\end{equation} where $b_i$ occurs odd number of times. %\red{itt nem kell feltenni, hogy $b_1 \ne b_2$?, bar ez kesobb jon}
\end{theorem}

\begin{proposition}\label{paqm}
Let $X$ be a totally ordered set and let $F\colon X^n\to X$ be an
associative, quasitrivial, monotone operation. Then $F$ is
reducible.
\end{proposition}
\begin{proof}
According to Theorem \ref{thmAkk}, if $F$ is not reducible, then $n$
is odd. Hence $n\ge 3$ and there exist $b_1, b_2$ satisfying
equation \eqref{eqbk}.
Since $b_1\ne b_2$, %, otherwise $b_1=b_2$ is a neutral element with respect to $F$. Moreover
we may assume that $b_1<b_2$ (the case $b_2<b_1$ can be handled
similarly). By the assumption \eqref{eqbk} for $b_1$ and $b_2$ we
have
\begin{equation}\label{eq:b1b2}
F(n \cdot b_1)=b_1, ~ F(b_2, (n-1) \cdot b_1)=b_2, ~ F(b_2,
(n-2)\cdot b_1, b_2)=b_1.
\end{equation}
By Lemma \ref{lASP}, $F$ is nondecreasing in the first and the last
variable. Thus we have $$F(n \cdot b_1) \le F(b_2, (n-1) \cdot
b_1)\le F(b_2, (n-2)\cdot b_1, b_2).$$ This implies $b_1=b_2$, a
contradiction.
\end{proof}

The following was proved as \cite[Corollary 4.9]{KS}.
\begin{corollary}\label{cormain}
Let $X$ be a nonempty chain and $n\ge 2$ be an integer. An
associative, idempotent, monotone operation $F:X^n\to X$ is
reducible if and only if
$F$ is nondecreasing.% \red{itt az van, hogy a ketvaltozosra belattuk ezt, es hivatkozni kene ra?}
\end{corollary}

Using Proposition \ref{paqm} and Corollary \ref{cormain} we get the
statement.

\begin{corollary}\label{cnemh}
Let $n\ge 2 \in \mathbb{N}$ be given, $X$ be a nonempty chain and
$F:X^n \to X$ be an associative quasitrivial operation. $$F \textrm{
is monotone } \Longleftrightarrow F \textrm{ is nondecreasing}.$$
\end{corollary}

\end{document}